\documentclass[10pt,reqno]{amsart}
\usepackage{amsmath,amsfonts,amssymb,amscd,amsthm,amsbsy,bbm, epsf,calc,enumerate,color,graphicx,verbatim,cite}
\usepackage{ifpdf}
\ifpdf
    \usepackage[colorlinks,citecolor=blue,linkcolor=blue]{hyperref}
\fi
\usepackage[shortalphabetic,initials]{amsrefs}
\usepackage{color}
\usepackage{datetime}
\usepackage[normalem]{ulem}

\newcounter{hours}\newcounter{minutes}

\makeatletter
\newtheorem*{rep@theorem}{\rep@title}
\newcommand{\newreptheorem}[2]{%
\newenvironment{rep#1}[1]{%
 \def\rep@title{#2 \ref{##1}}%
 \begin{rep@theorem}}%
 {\end{rep@theorem}}}
\makeatother

\theoremstyle{theorem}
\newtheorem{definition}{Definition}[section]
\newtheorem{theorem}{Theorem}[section]
\newreptheorem{theorem}{Theorem}

\newtheorem{lem}[theorem]{Lemma}
\newtheorem{lemma}[theorem]{Lemma}
\newreptheorem{lem}{Lemma}
\newtheorem{cor}[theorem]{Corollary}
\newtheorem{corollary}[theorem]{Corollary}

\newtheorem{proposition}[theorem]{Proposition}

\theoremstyle{definition}

\theoremstyle{remark}                  

\newtheorem{remark}[theorem]{Remark}

\def\A{{\mathcal A}}

\def\N{{\mathcal N}}

\def\R{{\mathbb R}}

\def\e{\varepsilon}

\def\Rn{\mathbb R^n}

\newcommand{\dist}{\operatorname*{dist}}

\newcommand{\abs}[1]{\left| #1 \right|}
\newcommand{\set}[1]{\left\{ {#1} \right\}}
\newcommand{\pth}[1]{\left( {#1} \right)}
\newcommand{\bra}[1]{\left[ {#1} \right]}

\newcommand{\norm}[1]{\lVert#1\rVert}

\definecolor{darkgreen}{rgb}{0,0.7,0}
\definecolor{grey}{rgb}{0.5,0.5,0.5}

\newcommand{\cl}[1]{\overline{#1}}
\newcommand{\dx}{\;dx}
\newcommand{\dt}{\;dt}
\newcommand{\dS}{\;dS}

\numberwithin{equation}{section}
\newcommand{\halfliminf}{\operatorname*{lim\,inf_*}}
\newcommand{\halflimsup}{\operatorname*{lim\,sup^*}}

\newcommand{\supers}{\overline{\mathcal{S}}}
\newcommand{\subs}{\underline{\mathcal{S}}}
\newcommand{\sol}{\mathcal{S}}
\newcommand{\strictordwrt}[1]{\prec_{#1}}
\newcommand{\notstrictordwrt}[1]{\not\prec_{#1}}

\newcommand{\trace}{\operatorname{trace}}
\newcommand{\interior}{\operatorname{int}}

\begin{document}
\title{Porous medium equation to Hele-Shaw flow with general initial density}
\author{Inwon Kim}
\address[I. Kim]{Department of Mathematics, UCLA, USA.}
\thanks{I. Kim was partially supported by NSF DMS-0970072. ikim@math.ucla.edu}
\author{Norbert Po\v{z}\'{a}r}
\address[N. Po\v{z}\'{a}r]{Faculty of Mathematics and Physics, Institute of Science and
Engineering, Kanazawa University, Japan.}
\thanks{N. Pozar was partially supported by JSPS KAKENHI Grant Number 26800068. npozar@se.kanazawa-u.ac.jp}
\date{}
\begin{abstract}
In this paper we study the ``stiff pressure limit" of the porous medium equation, where the initial density is a bounded, integrable function with a sufficient decay at infinity. Our particular model, introduced in \cite{PQV}, describes the growth of a tumor zone with a restriction on the maximal cell density. In a general context, this extends previous results of Caffarelli-Vazquez \cite{CV} and Kim \cite{K03} who restrict the initial data to be the characteristic function of a compact set.  In the limit a Hele-Shaw type problem is obtained, where the interface motion law reflects the acceleration effect of the presence of a positive cell density on the expansion of the maximal density (tumor) zone.

\end{abstract}
\maketitle

\section{Introduction}

In this paper we consider the following degenerate diffusion equation
\begin{equation}\label{pme}
\rho_t - \nabla\cdot(\rho \nabla p)  = \rho G(p) \qquad \text{ in } \R^n\times (0,\infty),
\end{equation}
where
\begin{align}
\label{Pm}
p = P_m(\rho) := \frac{m}{m-1}\rho^{m-1},
\end{align}
$G' <0$ and $G(p_M) = 0$ for some $p_M>0$. The model \eqref{pme} was introduced in \cite{PQV} as a
model problem which describes the growth of cancer cells, with focus on the mechanical  aspect of
the cell density motion. Here the pressure $p$ discourages the overgrowth of the cell density
$\rho$ over some critical density $\rho_c$, which is normalized here as $1$. In \cite{PQV} the
convergence of the solution $\rho$ of \eqref{pme} and the corresponding pressure variable $p$ was
studied in the {\it stiff pressure limit}, i.e., as $m\to\infty$, in the setting of the weak
solutions. In the model of a fluid flow, $m$ characterizes the compressibility of the fluid with $m
\to \infty$ representing the incompressible limit.
It is shown in \cite{PQV} in the $L^1$ setting that $\rho$ and $p$ converges to the limit functions $\rho_{\infty}$ and $p_{\infty}$, satisfying the following equations
\begin{equation}\label{1}
-\Delta p_{\infty} = G(p_{\infty})\hbox{ in } \Omega(t):=\{p_{\infty}(\cdot,t)>0\} =
\{\rho_{\infty}(\cdot, t)=1\},
\end{equation}
\begin{equation}\label{2}
(\rho_{\infty})_t - \nabla\cdot(\rho_{\infty}\nabla p_{\infty}) = \rho_{\infty}G(p_{\infty}) \hbox{ in } \R^n\times (0,\infty).
\end{equation}

\medskip

 We mention that, even at a formal level, it is not clear how to derive from \eqref{1}--\eqref{2} the velocity law of the free boundary of the tumor region, $\partial\{\rho_{\infty}=1\}$. In \cite{PQV} it was conjectured  that the normal velocity law \eqref{boundary} holds for general solutions. This is what we prove, along with the uniform convergence of the density variable away from the boundary of the tumor region. Roughly speaking we will show the following (see Theorem~\ref{main} below for the precise statements).

 \begin{itemize}
\item[(a)]As $m\to\infty$,  $\rho_m$ uniformly converges to $1$ inside $\Omega(t)$ and to $\rho_0e^{G(0)t}$ outside of $\overline{\Omega(t)}$,
\item[(b)] $\overline{\{\rho_{\infty}=1\}}$ equals the closure of $\cup_{t>0}(\Omega(t)\times\{t\})$,
\item[(c)] the set  $\Omega(t)$ evolves with the normal boundary velocity (in the viscosity solutions sense)
 \begin{equation}\label{boundary}
 V= \frac{|\nabla p_{\infty}|}{1-\min[1,\rho_0 e^{G(0)t}]} \quad \hbox{ on } \partial\Omega(t).
 \end{equation}
\end{itemize}

Note that (a) and (b) above imply that $\rho_0 e^{G(0) t}\leq 1$ outside of $\Omega(t)$, and thus
the term $\min[1,\rho_0 e^{G(0) t}]$ in (c) at a boundary point $x\in\partial\Omega(t)$ is the outer limit of $\rho$ from the complement of $\cl{\Omega(t)}$.  Thus \eqref{boundary} coincides with the velocity law conjectured in \cite{PQV}. See Theorem~\ref{main} for a more precise statement.

\medskip

Note that (c) indicates that $\rho$ is generically discontinuous across $\partial\Omega(t)$. Thus proving the convergence result requires keeping track of the pressure variable, which appears to be, at least when $\Omega(t)$ has smooth boundary, continuous across $\Omega(t)$. In terms of the pressure, the equation can be written as
\begin{equation}\label{pressure}
p_t = (m-1)p\Delta p + |\nabla p|^2+ (m-1)pG(p).
\end{equation}

Now to state our main result in precise terms, let us denote by $\rho_m$ and $p_m$ the (density and pressure) solutions of \eqref{pme}. We will show the convergence of $p_m$ as $m\to\infty$ to the viscosity solution of the following free boundary problem:

\begin{align}
\label{FB}
\tag{FB}
\left\{\begin{aligned}
-\Delta p &= G(p) &&\hbox{ in }\{p(\cdot,t)>0\},\\
V &= g(x,t)|\nabla p| &&\hbox{ on }\partial\{p(\cdot,t)>0\},\\
\{\rho^E=1\} &\subset\overline{\{p(\cdot,t)>0\}}.
\end{aligned}
\right.
\end{align}
Here $g(x,t):= \frac{1}{1-\min[1,\rho^E(x,t)]}$ is the free boundary velocity coefficient,
and $\rho^E(x,t):= \rho^E_0(x) e^{G(0)t}$ is the density in the ``exterior'' region.
We set $g = +\infty$ whenever $\rho^E \geq 1$.

\medskip

As for the initial data for the free boundary problem \eqref{FB}, we shall assume that
\begin{equation}\label{initial0}
\begin{aligned}
&\Omega_0 \subset \R^n \text{ open bounded}, \quad \partial\Omega_0 \in C^{1,1},\\
&\rho^E_0\in C(\R^n) \hbox{ with } 0\leq \rho^E_0 < 1 \hbox{ and }\rho^E_0 \to 0 \text{ as }
\abs x \to \infty.\\
\end{aligned}
\end{equation}
Note that $\rho^E_0$ is the initial density in the ``exterior'' region, that is, the region outside of
$\Omega_0$.

%
%

\subsection*{Initial data for $\rho_m$}

In terms of the density variable, we would like to show that $\rho_m$ converge to  $\rho(\cdot,t):=
\chi_{\Omega(t)} + \rho^E\chi_{\Omega(t)^c}$, where $\Omega(t)=\{p(\cdot,t)>0\}$. To this end we
will show that the convergence holds locally uniformly for a ``well-prepared" initial density
$\rho_{0,m}$ approximating the initial density function $\rho_0:= \chi_{\Omega_0} + \rho^E_0
\chi_{\Omega_0^c}$. Our approximation is constructed such that the corresponding solution $\rho_m$
is increasing in time (see Lemma~\ref{monotone}). As for the general initial data $\rho_{0,m}$
approximating $\rho_0$, the convergence then will hold in $L^1$ norm due to the convergence result
for  the specific $\rho_{0,m}$ (Theorem~\ref{main}) as well as the $L^1$ contraction inequality for
$\rho_m$ \eqref{contraction}. While we believe that the monotonicity of $\rho_m$ is not an
essential ingredient of the convergence proof in section 4, it is not clear at the moment whether
the uniform convergence result obtained in Theorem~\ref{main} holds for general choices of
$\rho_{0,m}$ (see Corollary~\ref{cor:general}).

\medskip

To construct our specific approximation $\rho_{0,m}$, let us first assume that  $\rho^E_{0,m}$ satisfies, for some $\delta>0$ which is independent of $m$,
\begin{equation}\label{initial}
\begin{aligned}
\rho^E_{0,m}\in L^1(\R^n) \cap C^{1,1}(\R^n), \quad 0 \leq \rho^E_{0,m} < 1 - \delta,\\
\rho^E_{0,m} \to \rho^E_0 \quad \text{ locally uniformly as $m \to \infty$,}\\
\quad m (1 - \delta)^m \norm{D^2 \rho^E_{0,m}}_\infty \to 0 \quad \text{as } m\to\infty.
\end{aligned}
\end{equation}
Next suppose
\begin{align}
\label{matched}
\rho_{0,m} := \max \left( P_m^{-1}(p_0), \rho_{0,m}^E \right),
\end{align}
where $P_m$ was introduced in \eqref{Pm}, and
$p_0$ is the unique smooth solution of
\begin{align*}
\left\{
\begin{aligned}
- \Delta p_0 &= G(p_0) && \text{in } \Omega_0,\\
p_0 &= 0 && \text{on } \R^n \setminus \Omega_0.
\end{aligned}
\right.
\end{align*}
As we shall see in the lemma below, this will guarantee that $\rho_m$ is monotone increasing in
time.  After we obtain convergence result for this particular approximation of $\rho_0$, we can use
$L^1$ contraction for solutions of \eqref{pme} to address the case of general $\rho_{0,m}$.

\begin{remark}
Given $\rho^E_0$ satisfying \eqref{initial0}, we can easily define $\rho^E_{0,m} = \rho^E_0 *
\eta_{1/m}$, where $\eta_{1/m}$ is the standard mollifier with radius $1/m$.  Such initial data
satisfies the assumptions \eqref{initial}.  Indeed, we can easily estimate $\norm{D^2
\rho^E_{0,m}}_\infty \leq \norm{\rho^E_0}_{L^1} \norm{D^2 \eta_{1/m}}_\infty \leq C m^2$. The rest
of \eqref{initial} is standard. These assumptions, as in \cite{PQV}, are required to prevent the jump singularity of $\rho_m$ over time at $t=0$.
\end{remark}

Let us now state the main result in this paper.

\begin{theorem}\label{main}
Let $\rho_m$ solve \eqref{pme} with $\rho_{0,m}$ satisfying \eqref{initial}-\eqref{matched}, and let $p_m$ be the corresponding pressure variable.  Then the following holds:
\begin{itemize}
\item[(a)] (Theorem~\ref{th:hs-well-posedness}) There is a unique viscosity solution $p$ of \eqref{FB} with initial data $p_0$, where $p_0$ solves $-\Delta p_0 = G(p_0)$ in $\Omega_0$, and zero otherwise;\\
\item[(b)] (Lemma~\ref{easy}(b)) $\{\rho^E \geq 1\}$ is contained in the closure of $\{p>0\}$;\\
\item[(c)] (Corollary~\ref{cor:main}) The pressure variable $p_m$ locally uniformly converges to $p$ as long as $p$ is continuous; \\
\item[(d)] (Corollary~\ref{cor:main}) $\rho_m$ locally uniformly converges to $\rho:= \chi_{\{p>0\}} +\rho^E \chi_{\set{p = 0}}$ away from $\partial\set{p>0}$.\\
\item[(e)] (Corollary~\ref{cor:00}) assuming that $\rho^E_0$ is a Lipschitz continuous function, $\partial\set{p>0}$ has zero Lebesgue measure in $\R^n\times [0,\infty)$. \\
 \item[(f)] (Proposition~\ref{BV}) $\partial\set{p(\cdot,t)>0}$ is of finite perimeter as long as $\rho^E(\cdot,t)<1$ on $\partial\set{p(\cdot,t)>0}.$
\end{itemize}
\end{theorem}

Note that the free boundary motion law in \eqref{FB} yields (a) a generic discontinuity of $\rho$ across $\partial\set{p>0}$ and (b) a generic discontinuity of $p$ over time when the region $\{\rho^E\geq1\}$ nucleates. For this reason the convergence of $\rho_m$ and $p_m$ as stated appears to be optimal.

\begin{remark}
Due to the fact that $\rho$ is nonzero outside of $\{p>0\}$, the set $\{p_m>0\}$ will degenerate as $m\to\infty$ and will not converge to $\{p>0\}$. But our result (Corollary~\ref{cor:main}) implies that  for any $\e>0$, the set $\{p_m>\e\}$ will be a subset of $\{p>0\}$ for sufficiently large $m$. In fact one can characterize $\{p>0\}$ as
$$
\{p>0\} = \{\liminf_{m\to\infty} p_m>0\}.
$$
\end{remark}

\medskip

As in \cite{K03} we will be using the notion of viscosity solutions, which is based on comparison principle with appropriate choices of test functions.  In our problem these will be radial functions in local neighborhoods with fixed boundaries. In the viscosity solutions theory, this  corresponds to the usage of second-order polynomials as test functions for nonlinear elliptic equations (see for instance \cite{CIL}). Therefore the first crucial step in the argument is to prove the above theorem in the radial case. When there is no surrounding density, i.e., when $\rho^E_0 = 0$, we rely on {\it Barenblatt solutions}, a well-known family of radially symmetric, compactly supported solutions of the porous medium equation. Based on the convergence of these radial solutions we apply the viscosity solution approach to obtain the corresponding result in \cite{K03}.  On the other hand, when $\rho^E_0$ is non-zero, there are no such explicit solutions available in the radial setting. The other challenges we face are the possible jump-type discontinuity over time of the tumor set $\{p>0\}$ due to the free boundary velocity becoming infinite in the law \eqref{boundary} when the density reaches one, as well as the source term $G(p)$, which each prevent the straightforward application of a comparison principle argument between subsolutions and supersolutions.

\medskip

\subsection*{Formal derivation of the free boundary motion law}

Before we finish this section let us present a formal computation indicating the free boundary velocity law \eqref{boundary}. Let us write \eqref{pme} as
$$
\rho_t - \Delta \tilde{p} = \rho G(p), \hbox{ where }\tilde{p} =\rho^m.
$$

Formally from the definition of $\tilde{p}$ it should  be clear that $\tilde{p}$ and the pressure variable converges to the same limit $p_{\infty}$ as $m\to\infty$. Let us also denote the limit density solution as $\rho_{\infty}$, and suppose that $\rho_{\infty}$ is discontinuous across $\Omega(t)=\{p_{\infty}(\cdot, t)>0\} = \{\rho_{\infty}(\cdot, t)=1\}$. Again if we take the time derivative of the total mass at the formal level, denoting $p_{\infty}=p$, $\rho_{\infty}=\rho$ and  $\rho^+$ and $\rho^-$ as $\rho_{\infty}$ inside and outside of $\Omega(t)$, then we have

\begin{align*}
\begin{aligned}
\int \rho G(p) = \frac d{dt} \int \rho \dx
&= \frac d{dt} \bra{ \int_{\Omega(t)} \rho \dx  + \int_{\R^n-\Omega(t)}  \rho \dx }\\
&= \int_{\Omega(t)} (\rho^+)_t \dx + \int_{\partial\Omega(t)} V (\rho^+ -\rho^-)\;dS + \int_{\R^n-\Omega(t)} (\rho^-)_t \dx\\
&=  \int_{\Omega(t)} \Delta p +\int_{\partial\Omega(t)} V (\rho^+ -\rho^-)dS +\int \rho G(p)\\
&= \int_{\partial\Omega(t)} [-|Dp| + V(\rho^+-\rho^-) ] dS + \int \rho G(p).
\end{aligned}
\end{align*}
This computation indicates \eqref{boundary}.
\vspace{20pt}

\subsection*{Outline}
In section 2 we will prove the comparison principle and uniqueness for the limiting free boundary problem \eqref{FB}. The main results are Theorem ~\ref{th:comparison} and Theorem~\ref{th:hs-well-posedness}.
They extend the comparison and well-posedness results from \cite{Pozar14}
for the Hele-Shaw problem with a time-dependent free boundary velocity coefficient $g$.
The main challenge is to allow for an infinite coefficient depending on time.
This is handled by a shift in time using the fact that the coefficient is nondecreasing in time and possesses a certain regularity.
In section 3 we show the convergence in the radially symmetric setting with fixed boundary data.
Let us mention that we rely on a compactness argument based on integral estimates to derive the convergence of the radial solutions in local neighborhoods. Direct derivation of convergence using barriers is an interesting open question at the moment.  Our integral estimates are modified versions from \cite{PQV} due to the presence of fixed boundaries.
In section 4 we prove the convergence result (Corollary~\ref{cor:main}) based on the comparison principle in section 2 as well as the radial convergence result in section 3. Lastly, in section 5 we present an estimate on the perimeter of the set $\{p>0\}$ based on geometric arguments.

\begin{remark}
Before completion of this paper we learned that similar results were shown by Mellet, Perthame and Quir\'{o}s \cite{MPQ} following a different approach. Their approach
relies on integral estimates, while  ours relies on pointwise arguments which yield uniform convergence results. We believe that both of our approaches have different merits for applications to different contexts.
\end{remark}
\section{Notion of solutions and comparison principle}

\subsection{Notation}

We will follow the notation from \cite{Pozar14}.

Let $E \subset \R^d$ for some $d \geq 1$.
Then $USC(E)$ and $LSC(E)$ are respectively
the sets of all upper semi-continuous
and lower semi-continuous functions on $E$.
For a locally bounded function $u$ on $E$
we define the semi-continuous envelopes
$u^{*, E} \in USC(\R^d)$ and $u_{*,E} \in LSC(\R^d)$
as
\begin{align}\label{usclsc}
\begin{aligned}
    u^{*,E} &:=
    \inf \set{v \in USC(\R^d): v \geq u \text{ on } E},\\
    u_{*, E} &:= \sup \set{v \in LSC(\R^d): v \leq u \text{ on } E}.
\end{aligned}
\end{align}
Note that $u^{*,E} : \R^d \to [-\infty, \infty)$
and $u_{*, E} : \R^d \to (-\infty, \infty]$ are finite on $\cl E$.
We simply write $u^*$ and $u_*$
if the set $E$ is understood from the context.
The envelopes can be also expressed as
\[
u^{*,E}(x) = \lim_{\delta\to0} \sup \set{u(y) : y \in E,\ \abs{y - x} < \delta}
\quad \text{for } x \in \cl E,
\qquad u_{*,E} = -(-u)^{*,E}.
\]

Let us review the shorthand notation for
the set of positive values of a given function $u: E \to \R$,
defined on a set $E \subset \Rn\times \R$,
\begin{equation*}
\Omega(u; E) := \set{(x,t) \in E: u(x,t) > 0}, \qquad
\Omega^c(u; E) := \set{(x,t) \in E: u(x,t) \leq 0},
\end{equation*}
and the closure $\cl\Omega(u;E) := \cl{\Omega(u;E)}$.
For $t \in \R$, the time-slices $\cl\Omega_t(u; E)$,
$\Omega_t(u; E)$ and $\Omega^c_t(u; E)$
are defined in the obvious way, i.e.,
\begin{align*}
\cl\Omega_t(u;E) = \set{x : (x,t) \in \cl\Omega(u;E)}, \qquad \text{etc.}
\end{align*}
We shall call the boundary of the positive set in $E$ the \emph{free
boundary} of $u$ and denote it $\Gamma(u;E)$, i.e.,
\begin{align*}
\Gamma(u; E) = (\partial \Omega(u; E)) \cap E.
\end{align*}
If the set $E$ is understood from the context, we shall simply write
$\Omega(u)$, etc.

For given constant $\tau \in \R$ we will often abbreviate
\begin{align*}
\set{t \leq \tau} := \set{(x,t) \in \Rn\times\R: t \leq \tau},
    \qquad \text{etc.}
\end{align*}

\subsection{Viscosity solutions}

We will consider a general problem for the introduction of the notion of viscosity solutions.
To be more specific, we will define solutions of the problem
\begin{align}
\label{hsg}
\left\{
\begin{aligned}
F(D^2u, Du, u) &= 0 && \text{in } \set{u > 0},\\
u_t - g \abs{Du}^2 &= 0  && \text{on } \partial\set{u > 0}.
\end{aligned}
\right.
\end{align}
We assume that
$F$ is a general elliptic operator $F(D^2u, Du, u)$ that satisfies the following:
There exist constants $c_0$, $c_1 \geq 0$ and $0 < \lambda \leq \Lambda$ such that
\begin{align*}
\mathcal P_{\lambda, \Lambda}^-(M - N) - c_1\abs{p -q} - c_0 \abs{z - w}
&\leq F(M, p, z) - F(N, q, w)\\
&\leq \mathcal P_{\lambda, \Lambda}^+(M - N) + c_1\abs{p -q} + c_0 \abs{z - w}
\end{align*}
for all $M, N \in Sym_n$, $p,q \in \Rn$, $z,w \in \R$,
where $\mathcal P_{\lambda, \Lambda}^\pm$
are the Pucci extremal operators.
This guarantees that $F$ has the strong maximum principle and Hopf's lemma;
see \cite{ArmstrongThesis}.
Then we need to assume that $F_u > 0$ and that for some $p_M > 0$
\begin{align*}
F(0,0,0) < 0 \qquad \text{and} \qquad F(0,0,p_M) = 0.
\end{align*}

\begin{remark}
In the case of \eqref{FB} we set $F(X, p, u) = - \trace X - G(u)$.
\end{remark}
For the velocity coefficient $g: \Rn\times\R \to (0,\infty]$ we will assume that
\begin{align}
\label{g-assumption}
\begin{aligned}
&\text{$g$ is continuous at every point of $\set{g < \infty}$}\\
\text{and } &\text{$g(\hat x, \hat t) = \liminf_{(x,t)\to(\hat x, \hat t)} g(x,t)$ for all $(\hat x, \hat t)$}
\end{aligned}
\end{align}

As in the previous papers \cite{K03, Pozar14}, we define viscosity solutions in two ways: using barriers and using test functions.
These two notions will be shown to be equivalent, but each has its advantages in certain arguments.
We will use the notion using barriers, but we still include the notion via test functions to show the relation with the original definition in \cite{K03}.
The main difference from \cite{Pozar14} is to allow for $g = +\infty$.

Before proceeding with the definition of a viscosity solution,
we first recall the definition of parabolic neighborhood and strict separation used in \cite{Pozar14}:

\begin{definition}[Parabolic neighborhood and boundary]
\label{def:parabolic-nbd}
\ \\A nonempty set $E \subset \Rn\times \R$ is called a \emph{parabolic neighborhood} if $E = U \cap \set{t \leq \tau}$ for some open set $U \subset \Rn\times \R$ and some $\tau \in \R$.
We say that $E$ is a parabolic neighborhood of $(x,t) \in \Rn\times\R$
if $(x,t) \in E$.
Let us define $\partial_P E := \cl{E} \setminus E$, the \emph{parabolic boundary} of $E$.
\end{definition}

Now we introduce an important concept in the theory, the notion of \emph{strict separation}.
We shall use the version introduced in \cite{Pozar14}, which differs slightly from the one introduced in \cite{K03}).

\begin{definition}[Strict separation]
\label{def:strict-separation}
Let $E \subset \Rn \times \R$ be a parabolic neighborhood,
and $u, v : E \to \R$ be
bounded functions on $E$,
and let $K \subset \cl E$.
We say that $u$ and $v$ are strictly separated
on $K$ with respect to $E$,
and we write $u \strictordwrt{E} v$ in $K$,
if
$$
u^{*,E} < v_{*,E} \hbox{ in }
    K \cap \cl\Omega(u;E).
$$
\end{definition}

\begin{remark}
We do not require non-negative functions above, since taking a semicontinuous envelope commutes with taking the positive part and $0 \leq u^{*,E} = (u_+)^{*,E} = \pth{u^{*,E}}_+$ on $\cl\Omega(u;E)$.
\end{remark}

The following lemma was proved in \cite{Pozar14}.

\begin{lemma}[{c.f. \cite[Lemma~2.14]{Pozar14}}]
\label{le:cross}
Suppose that $E$ is a bounded parabolic neighborhood and $u$, $v$ are locally bounded functions on $E$.
The set
\begin{align}
\label{cross-time-set}
\Theta_{u,v;E} := \set{\tau: u \strictordwrt{E} v \text{ in } \cl E \cap \set{t \leq \tau}}
\end{align}
is open and $\Theta_{u,v;E} = (-\infty, T)$ for some $T \in (-\infty, \infty]$.
\end{lemma}

\subsubsection{Notion via barriers}

We build strict barriers for \eqref{hsg}.

\begin{definition}
\label{def:barrier}
Let $U \subset \Rn \times \R$ be a nonempty open set and let $\phi \in C^{2,1}(U)$ be such that $D\phi \neq 0$ on $\Gamma(\phi; U)$.
We say that $\phi$ is a \emph{sub-barrier} of \eqref{hsg} in $U$ if
there exists a positive constant $\delta > 0$ such that
\begin{enumerate}[(i)]
\item $F(D^2\phi, D\phi, \phi) < - \delta$ in $\Omega(\phi; U)$,
\item $\phi_t - g \abs{D\phi}^2 < - \delta$ on $\Gamma(\phi; U)$.
\end{enumerate}
A \emph{superbarrier} is defined analogously by reversing the inequalities in (i)--(ii) and the sign in front of $\delta$,
and requiring additionally that $g < \infty$ on $\Omega^c(\phi; U)$.
\end{definition}

Note that it is enough to consider barriers with finite free boundary velocity
since we will explicitly require in the definition that the positive set of a viscosity solution
that the positive set of the solution always contains the set where the free boundary
velocity coefficient $g$ is infinite.

\begin{remark}
The Definition~\ref{def:barrier} does not assume $\phi \geq 0$, we can always take the positive part later, as needed.
This does not seem to play a role in the strict separation in Definition~\ref{def:strict-separation}.
\end{remark}

The definition of solutions follows.

\begin{definition}
\label{def:visc-barrier}
We say that a locally bounded, non-negative
function $u: Q \to [0,\infty)$ is
a \emph{viscosity subsolution} of \eqref{hsg} on $Q$ if
for every bounded parabolic neighborhood $E \subset Q$
and every superbarrier $\phi$ on $U$
such that $u \strictordwrt{E} \phi$ on $\partial_P E$,
we also have $u \strictordwrt{E} \phi$ on $\cl E$.

Similarly, a locally bounded, non-negative
function $u: Q \to [0,\infty)$ is
a \emph{viscosity supersolution} of \eqref{hsg} if
$\set{g = \infty} \cap Q \subset \cl\Omega(u_*; Q)$,
and for every bounded parabolic neighborhood $E \subset Q$
and every subbarrier $\phi$ on $U$ such that $\phi \strictordwrt{E} u$ on $\partial_P E$, we also have $\phi \strictordwrt{E} u$ on $\cl E$.

Finally, $u$ is a \emph{viscosity solution} if it is both
a viscosity subsolution and a viscosity supersolution.
\end{definition}

\begin{remark}
\label{re:super-phase-containment}
Since we require $\set{g = \infty} \subset \cl\Omega(v)$ for all $v \in \supers(g; Q)$, we also have to address the stability of this. That is,
\begin{align*}
\set{g = \infty} \subset \cl\Omega(\inf_{v \in \mathcal A} v)
\end{align*}
whenever $\mathcal A \subset \supers(g; Q)$.
We need that $\set{g = \infty} = \cl{\interior \set{g = \infty}}$ for this.
Then we use subsolutions of the elliptic problem in the interior of the positive phase; they give uniform lower bound.
\end{remark}

\begin{remark}
It is not hard to check that if $u$ ($v$) is a viscosity sub(super)solution of \eqref{hsg} then
$$
-\Delta u \leq G(u) \quad \hbox{ and } -\Delta v \geq G(v) \hbox{ in } \{v(\cdot,t)>0\}.
$$
\end{remark}

\begin{remark}
As is standard in the viscosity theory, it is enough to consider only simple cylinders with balls as their base as the parabolic neighborhoods $E$ in Definition~\ref{def:visc-barrier}.
\end{remark}

\subsubsection{Notion via test functions}

Similarly to the previous work in \cite{K03,Pozar14}, we can give an equivalent definition of the notion of viscosity solutions via test functions.
In the following definitions, $Q$ is an arbitrary nonempty parabolic neighborhood.

\begin{definition}
\label{def:visc-test-sub}
We say that a locally bounded, non-negative
function $u: Q \to [0,\infty)$ is
a \emph{viscosity subsolution} of \eqref{hsg} on $Q$ if
\begin{enumerate}[(i)]
\item \emph{(continuous expansion)}
\[
\cl\Omega(u; Q) \cap Q \cap \set{t \leq \tau}
    \subset \cl{\Omega(u;Q) \cap \set{t < \tau}}
    \cup \set{g = \infty}
    \quad \text{for every $\tau > 0$},
\]

\item \emph{(maximum principle)}\\
for any $\phi \in C^{2,1}$
such that $u^* - \phi$ has a local maximum at
$(\hat x, \hat t) \in Q \cap \cl\Omega(u;Q)$
in $\cl\Omega(u;Q) \cap \set{t \leq \hat t}$, we have
\begin{enumerate}[({ii}-1)]
\item if $u^*(\hat x, \hat t) > 0$ then $F(D^2 \phi(\hat x, \hat t), D^2 \phi(\hat x, \hat t), u^*(\hat x, \hat t))\leq 0$,
\item if $u^*(\hat x, \hat t) = 0$ then either
    $F(D^2 \phi(\hat x, \hat t), D^2 \phi(\hat x, \hat t), 0) \leq 0$ or
    $D \phi(\hat x, \hat t) = 0$
    or
    $\phi_t(\hat x, \hat t) - g(\hat x, \hat t) \abs{D\phi}^2(\hat x, \hat t) \leq 0$.
\end{enumerate}
\end{enumerate}
\end{definition}

\begin{remark}
The condition (i) in Definition~\ref{def:visc-test-sub} is necessary
to prevent a scenario where a ``bubble'' closes instantly;
more precisely, a subsolution cannot become instantly
positive on an open set surrounded by a positive phase,
or cannot fill the whole space instantly,
unless the expansion of the positive phase happens into the set $\set{g = \infty}$.
\end{remark}

\begin{definition}
\label{def:visc-test-super}
We say that a locally bounded, non-negative
function $u: Q \to [0,\infty)$ is
a \emph{viscosity supersolution} of \eqref{hsg} on $Q$ if
\begin{enumerate}[(i)]
\item (support)
\begin{enumerate}[({i}-1)]
\item
    if $(\xi,\tau) \in \Omega(u_*; Q)$ then
    $(\xi, t) \in \Omega(u_*;Q)$
    for all $(\xi,t) \in Q$, $t \geq \tau$.
\item
    \begin{align*}
    \set{g = \infty} \cap Q \subset \cl\Omega(u_*; Q).
    \end{align*}
\end{enumerate}
\item (maximum principle)\\
for any $\phi \in C^{2,1}$
such that $u_* - \phi$ has a local minimum at $(\hat x, \hat t) \in Q$
in $\set{t \leq \hat t}$, we have
\begin{enumerate}[({ii}-1)]
\item if $u_*(\hat x, \hat t) > 0$ then $F(D^2 \phi(\hat x, \hat t), D^2 \phi(\hat x, \hat t), u_*(\hat x, \hat t)) \geq 0$,
\item if $u_*(\hat x, \hat t) = 0$ then either
    $F(D^2 \phi(\hat x, \hat t), D^2 \phi(\hat x, \hat t), 0) \geq 0$ or
    $D \phi(\hat x, \hat t) = 0$
    or
    $g(\hat x, \hat t) < \infty$ and
    $\phi_t(\hat x, \hat t) - g(\hat x, \hat t) \abs{D\phi}^2 (\hat x, \hat t) \geq 0$.
\end{enumerate}
\end{enumerate}
\end{definition}

\begin{remark}
As was noted in \cite{Pozar14}, assumption Definition~\ref{def:visc-test-super}(i-1) is there only to make our life easier.
\end{remark}

\begin{remark}
The closure in the condition Definition~\ref{def:visc-test-super}(i-2) cannot be removed since $\Omega(u_*; Q)$ is a (relatively) open set.
If at a given time $g$ becomes $+\infty$ on an open set outside of $\Omega_t(u_*)$ in the previous times, then $u_*$ is zero on this set.
\end{remark}

\begin{remark}
As is standard in the theory of viscosity solutions, we can require that the test functions $\phi$ are smooth, even polynomials of at most second order in space and first order in time.
For (ii-2) we can use only radially symmetric test functions.
\end{remark}

The definition of a viscosity solution follows.

\begin{definition}
\label{def:visc-test}

We say that a locally bounded, non-negative
function $u: Q \to [0,\infty)$ is
a \emph{viscosity solution} of \eqref{hsg} on $Q$ if it is both viscosity subsolution and viscosity supersolution on $Q$.
\end{definition}

\subsection{Equivalence of notions}

We now get a result similar to \cite[Proposition~2.13]{Pozar14}.

\begin{proposition}
\label{pr:visc-equivalency}
The definitions of viscosity subsolutions (resp. supersolutions)
in Definition~\ref{def:visc-test-sub} (resp. \ref{def:visc-test-super})
and in Definition~\ref{def:visc-barrier}
are equivalent.
\end{proposition}

\begin{proof}
The direction from Definition~\ref{def:visc-test-sub} follows the proof of
\cite[Proposition~2.13]{Pozar14}.
The only detail that we have to check is that the supports of a subsolution and a superbarrier
stay ordered at the crossing time.
Since the continuous expansion of subsolution in Definition~\ref{def:visc-test-sub}(i)
is valid only in the set $\set{g < \infty}$,
we need to use the fact that a for superbarrier in Definition~\ref{def:barrier}
satisfies $\Omega^c(\phi; U) \subset \set{g < \infty}$.

We do not have this issue with supersolutions, so the proof is standard.

The direction from Definition~\ref{def:visc-barrier} to Definition~\ref{def:visc-test-sub}
and \ref{def:visc-test-super} is also standard.
The continuous expansion Definition~\ref{def:visc-test-sub}(i) can be verified by a comparison
with radially symmetric barriers.
The monotonicity of the support of a supersolution Definition~\ref{def:visc-test-super}(i-1),
an open set at every time,
can be shown by a comparison with a stationary subbarrier such as
$\phi(x,t) = \alpha(c - \abs x^2)_+$ for appropriate constants $\alpha, c > 0$.
\end{proof}

With this proposition, we will from now on use the two notions of subsolutions
and supersolutions
from Definition~\ref{def:visc-barrier}, and from
Definition~\ref{def:visc-test-sub}
and \ref{def:visc-test-super}
interchangeably.

\subsection{Viscosity solution classes}

\begin{definition}
For a given function $g$
and a nonempty parabolic neighborhood $Q \subset \Rn \times \R$ and $g$ satisfying \eqref{g-assumption}
we define the following classes of solutions:
\begin{itemize}
\item $\supers(g,Q)$,
    the set of all viscosity supersolutions of the Hele-Shaw problem
    \eqref{hsg} on $Q$;
\item $\subs(g,Q)$,
    the set of all viscosity subsolutions of \eqref{hsg} on $Q$;
\item $\sol(g,Q) = \supers(g,Q) \cap \subs(g,Q)$,
    the set of all viscosity solutions of \eqref{hsg} on $Q$.
\end{itemize}
\end{definition}

\subsection{Basic properties of solutions}

A subsolution is a subsolution of the elliptic problem on the whole space.

\begin{proposition}
\label{pr:subs-elliptic}
If $u \in \subs(g, Q)$ for some $g$ and $Q$ then $x \mapsto u^*(x, \hat t)$ is the standard viscosity solution of
\begin{align*}
F(D^2 \psi, D\psi, \psi) \leq 0.
\end{align*}
on $\set{x: (x,\hat t) \in Q}$ for every $\hat t \in \R$.

Similarly, if $u \in \supers(g, Q)$ for some $g$ and $Q$, then $x \mapsto u_*(x, \hat t)$ is the standard viscosity solution of
\begin{align*}
F(D^2 \psi, D\psi, \psi) \geq 0.
\end{align*}
on $\Omega_{\hat t}(u_*, Q)$.
\end{proposition}

\begin{proof}
The proof is analogous to the proof of \cite[Lemma~3.3]{KP13}.
\end{proof}

\subsection{Comparison principle}

\begin{theorem}
\label{th:comparison}
Let $Q$ be a bounded parabolic neighborhood and
let $g_1$ and $g_2$ be two velocity coefficients satisfying \eqref{g-assumption} for which there exists $\hat r > 0$ such that
\begin{align}
\label{conv-order}
\overline g(x,t) := \sup_{\cl B_{\hat r}(x,t) \cap Q} g_1 \leq \inf_{\cl B_{\hat r}(x,t) \cap Q} g_2 =: \underline g(x,t) \qquad \text{for all $(x,t) \in Q$.}
\end{align}
If $u \in \subs(g_1, Q)$ and $v\in \supers(g_2, Q)$
such that $u \strictordwrt{Q} v$ on $\partial_P Q$,
then $u \strictordwrt{Q} v$ in $\cl Q$.
\end{theorem}

\subsection{Proof of comparison principle}

We can assume that $u \in USC(\cl Q)$ and $v \in LSC(\cl Q)$.

We would like to follow the proof of \cite[Theorem~2.18]{Pozar14}.
We will use the assumption \eqref{conv-order} to justify the use of sup- and inf-convolutions.

The structure of the proof is similar to the previous papers \cite{K03,KP13,Pozar14}, with minor modifications to allow for the unbounded velocity coefficient.
We first regularize the free boundaries of $u$ and $v$ by means of the sup- and inf-convolutions over a set of particular shape to guarantee interior/exterior ball property in both space and space-time.
The set for inf-convolution is decreasing in time to add an additional perturbation, by effectively increasing the free boundary velocity of the supersolution.
Now, if the comparison fails, the regularized solutions must cross.
We first show that due to the continuous expansion of the support of $u$, and the fact that $u$ and $v$ are sub/supersolutions of the elliptic problem, this crossing must happen on the free boundary.
At the first contact point, the boundaries are locally $C^{1,1}$ in space.
Moreover, the velocity coefficient $g_1$ for the subsolution is bounded on the neighborhood of this point.
At the regular contact point it is possible to define weak normal derivatives of the regularized solutions, which must be ordered by Hopf's lemma.
Moreover, we can construct barriers to show that the free boundary velocity law is satisfied with these weak normal derivatives.
An ordering of the free boundary velocities at the crossing point with the additional perturbation above then yields a contradiction.
Therefore the comparison holds.

Let us define the crossing time
\begin{align}
\label{cross-time-uv}
t_0 := \sup \Theta_{u, v; Q},
\end{align}
using the set $\Theta_{u,v;Q}$ defined in \eqref{cross-time-set}.
We observe that $u \strictordwrt{Q} v$ in $\cl Q$ is equivalent to $t_0 = \infty$.

Let us therefore suppose that $t_0 < \infty$ and we will show that this leads to a contradiction.

\subsubsection{Regularization}

We shall use the standard sup/inf-convolutions to regularize the free boundaries at the contact point.
We first introduce the open set $\Xi_r(x,t) \subset \Rn\times\R$ for $(x,t) \in \Rn\times\R$ and $r > 0$ as
\begin{align*}
    \Xi_r(x,t)
    = \set{(y,s) : (\abs{y - x}-r)_+^2 + \abs{s - t}^2 < r^2}.
\end{align*}
Note that $\cl\Xi_r(x,t) \subset B_{\hat r}(x,t)$ if $2r < \hat r$.

Let $T > 0$ be such that $Q \subset \set{t \leq T}$.
For given $0 < r < \hat r/2$, $0 < \delta < \frac r{2T}$ we define
\begin{align*}
Z(x,t) &=  \sup_{\cl \Xi_r(x,t)}u,\\
W(x,t) &= \inf_{\cl \Xi_{r - \delta t}(x,t)} v
\end{align*}
for $(x,t) \in Q_r$ with
\begin{align*}
Q_r := \set{(x,t) \in Q: \cl \Xi_r(x,t) \subset Q}.
\end{align*}
Note that $Q_r$ is a parabolic neighborhood.

The following lemma is standard.

\begin{lemma}
\label{le:W-Z}
For all $r, \delta > 0$ sufficiently small, $Z \in USC(Q_r)$, $W \in LSC(Q_r)$, and
\begin{align*}
Z \strictordwrt{Q_r} W \text{ on } \partial_P Q_r.
\end{align*}
For every $(x,t) \in Q_r$ there exists $(x_u, t_u) \in \cl \Xi_r(x,t) \subset Q$ and $(x_v, t_v) \in \cl \Xi_{r - \delta t}(x,t)$ such that
\begin{align*}
u(x_u, t_u) = Z(x,t) \qquad \text{and} \qquad v(x_v, t_v) = W(x,t).
\end{align*}
Moreover, $x \mapsto Z(x,t)$ is a subsolution of the elliptic problem on $\set{x : (x,t) \in Q_r}$ and $x \mapsto W(x,t)$ is a supersolution of the elliptic problem on $\Omega_t(W; Q_r)$.

The support of $Z$ expands continuously in the sense
\begin{align*}
\cl\Omega(Z; Q_r) \cap Q_r \cap \set{t \leq \tau} \subset \cl{\Omega(Z; Q_r) \cap \set{t < \tau}} \cup \set{\overline g = \infty} \qquad \text{for every } \tau >0.
\end{align*}

Similarly, the support of $W$ is nondecreasing,
\begin{align*}
\text{if $(\xi,\tau) \in \Omega(W; Q_r)$ then $(\xi, t) \in \Omega(W; Q_r)$ for all $(\xi, t) \in Q_r$, $t \geq \tau$,}
\end{align*}
and
\begin{align}
\label{g-inf-W-inter}
\set{\underline g = \infty} \cap Q_r \subset \Omega(W; Q_r)
\end{align}
\end{lemma}

\begin{remark}
We can prove a stronger result that actually $Z \in \subs(\overline g; Q_r)$ and $W \in \supers(\underline g; Q_r)$, where $\overline g$ and $\underline g$ are sup/inf of $g$ over $\cl \Xi_r$, but we actually never need this.
\end{remark}

\begin{proof}
The semicontinuity and existence of points $(x_u, t_u)$ and $(x_v, t_v)$ is standard from semicontinuity of $u$ and $v$.
We can choose $r < \hat r/2$ and $\delta < \frac T{2r}$ sufficiently small so that $Z$ and $W$ are strictly ordered on $\partial_P Q_r$ since $u$ and $v$ are strictly ordered on $\partial_P Q$.

To check that $x \mapsto Z(x,t)$ and $x \mapsto W(x,t)$ are a subsolution and a supersolution of the elliptic problem in $\set{x : (x,t) \in Q_r}$ and $\Omega_t(W; Q_r)$, respectively,
for every $t \in \R$, we just need to recall that
they are the supremum of subsolutions,
respectively the infimum of supersolutions,
of the elliptic problem due to Proposition~\ref{pr:subs-elliptic}

The continuous expansion of $Z$ follows from the continuous expansion of $u$.
Indeed, if $(\xi, \tau) \in \cl\Omega(Z; Q_r) \cap Q_r$
and $\overline g(\xi, \tau) < \infty$,
then $g_1 < \infty$ on $B_{\hat r}(\xi, \tau)$.
Moreover, there exists
$(\xi_u, t_u) \in \cl\Omega(u; Q) \cap \cl \Xi_r(\xi, \tau) \subset B_{\hat r}(\xi, \tau)$.
By the continuous expansion of $u$, we have
\begin{align*}
(\xi_u, t_u) \in \cl{\Omega(u; Q) \cap \set{t < t_u}}.
\end{align*}
By the definition of the sup-convolution, we conclude that
\begin{align*}
(\xi, \tau) \in \cl{\Omega(Z; Q_r) \cap \set{t < \tau}}.
\end{align*}

To see that the support of $W$ is nondecreasing, suppose that $(\xi, \tau) \in \Omega(W; Q_r)$.
Then by definition $\cl\Xi_{r- \delta \tau}(\xi, \tau) \subset \Omega(v; Q)$.
Since $v$ is a supersolution, its support is nondecreasing, Definition~\ref{def:visc-test-super}(i-1),
and therefore $\cl\Xi_{r - \delta t}(\xi, t) \subset \cl\Xi_{r - \delta \tau}(\xi, t) \subset \Omega(v; Q)$
for all $t \geq \tau$.
We conclude that $(\xi, t) \in \Omega(W; Q_r)$ for all $t \geq \tau$.

Finally, if $(\xi, \tau) \in Q_r$ with $\underline g(\xi, \tau) = \infty$,
then $g_2 = \infty$ on $\cl B_{\hat r}(\xi, \tau)$.
By Definition~\ref{def:visc-test-super}(i-2) we have $B_{\hat r}(\xi, \tau) \subset \cl\Omega(v; Q)$.
Therefore $B_\rho(\xi, \tau) \cap Q_r \in \cl\Omega(W; Q_r)$ for small $\rho > 0$
such that $B_{\hat r - \rho}(\xi, \tau) \supset \cl \Xi_{r - \delta \tau}(\xi, \tau)$.
\end{proof}

\subsubsection{Contact}

Let us define the contact time
\begin{align*}
\hat t := \sup\Theta_{Z, W; Q_r} < t_0 < \infty
\end{align*}
where $t_0$ was introduced as the crossing time in \eqref{cross-time-uv}.
We will show that this leads to a contradiction.

\begin{lemma}
\label{zw-contact-time}
$Z = W = 0$ on $\Omega^c_{\hat t}(W; Q_r)$ and $Z < W$ on $\Omega_{\hat t}(W; Q_r)$.
In particular, $Z \leq W$ on $Q_r \cap \set{t \leq \hat t}$.
\end{lemma}

\begin{proof}
Let us denote
\begin{align*}
z(x) := Z(x, \hat t), \qquad w(x) := W(x, \hat t)
\end{align*}
for $x \in D := \set{x : (x, \hat t) = Q_r}$.
Recall that $z$ and $w$ are a subsolution and a supersolution, respectively,
of the elliptic problem by Proposition~\ref{pr:subs-elliptic}.
The set $V := \Omega_{\hat t}(W; Q_r)$ is open, and has an exterior ball of radius $r/2$ at every point of its boundary.
By \eqref{g-inf-W-inter}, $\overline g \leq \underline g < \infty$ on $D \setminus V$.
We know from the definition of the contact time that
$\cl\Omega(Z; Q_r) \cap Q_r \cap \set{t < \hat t} \subset \Omega(W; Q_r)$.
Let $y$ be such that $B_{r/2}(y) \subset V^c$, we must have $z = 0$ on $B_{r/2}(y) \cap D$
by the continuous expansion of the support of $Z$ and the monotonicity of the support of $W$
in Lemma~\ref{le:W-Z} and \eqref{g-inf-W-inter}.
$z$ is a subsolution of the elliptic problem
and therefore $z = 0$ on $\cl B_{r/2}(y) \cap D$.
By covering $D \setminus V$ by such balls, we conclude that  $z = 0$ on $D \setminus V$.

Let $\hat x \in \cl V$ such that $z(\hat x) \geq w(\hat x)$.
We only need to prove that $\hat x \in \partial V$, and the conclusion then follows.
Let $U$ be the connected component of $V$ for which $\hat x \in U$.

We know that $z = 0$ on $\partial U \cap D$ from above, and therefore $z \leq w$ on $\partial U$.
If $\cl U \times \set{\hat t} \cap \partial_P Q_r \neq 0$,
then the strong maximum principle for the elliptic problem
implies that $z < w$ on $U$, a contradiction.

If $\cl U \times \set{\hat t} \subset Q_r$, we have to give a different argument.
Let $y \in U$ be a point of maximum of $z$ on $\cl U$.
Clearly $z(y) > 0$.
By the interior ball property, there exists $\xi$ such that $y \in B_r(\xi)$ and $z = z(y)$ on $B_r(\xi)$.
Since $\psi = c$ for $c \geq p_M$ is a supersolution of the elliptic problem on $U$, the strong maximum principle implies $z(y) < p_M$.
In particular, $z$ is a strict subsolution of the elliptic problem on $B_r(\xi)$.
We therefore cannot have $w \equiv z$ on $B_r(\xi)$.
We conclude that $z < w$ on $U$ by the strong maximum principle.
\end{proof}

We know from Lemma~\ref{le:cross} that $Z \notstrictordwrt{Q_r} W$
in $\cl{Q_r} \cap \set{t \leq \hat t}$
.
Therefore due to Lemma~\ref{zw-contact-time} we can find
\begin{align*}
(\hat x, \hat t) \in \cl\Omega(Z; Q_r) \cap \Omega^c(W; Q_r).
\end{align*}

Due to Lemma~\ref{le:W-Z} there exist points
\begin{align*}
(x_u, t_u) \in \partial\Xi_r(\hat x, \hat t) \cap \partial\Omega(u; Q) \quad \text{and} \quad (x_v, t_v) \in \partial\Xi_{r - \delta\hat t}(\hat x, \hat t) \cap \partial\Omega(v; Q).
\end{align*}
We have $\cl \Xi_r(\hat x, \hat t) \subset \Omega^c(u)$ and $\Xi_r(x_u, t_u) \cap Q_r \subset \Omega(Z)$.
Since $Z\leq W$ on $Q_r \cap \set{t \leq \hat t}$, we have
\begin{align*}
\cl\Xi_{r - \delta t}(x, t) \subset \Omega(v) \qquad \text{for } (x,t) \in \Xi_r(x_u, t_u) \cap \set{t \leq \hat t}.
\end{align*}

By ordering we have
\begin{align*}
\cl \Xi_r(x_u, t_u) \cap \cl \Xi_{r - \delta\hat t}(x_v, t_v) \cap \set{t \leq \hat t} = \set{(\hat x, \hat t)}.
\end{align*}

\subsubsection{Free boundary velocity}

Let $m_Z \in [-\infty, \infty]$ denote the normal velocity of $\partial \Xi_r(\hat x, \hat t)$ at $(x_u, t_u)$, which can be expressed as
\begin{align*}
m_Z = \frac{t_u - \hat t}{\sqrt{r^2 - \pth{t_u - \hat t}^2}}.
\end{align*}
Let us define the set
\begin{align*}
E := \bigcup_{\substack{(x,t) \in \Xi_r(x_u, t_u)\\t \leq \hat t}} \Xi_{r - \delta t} (x,t).
\end{align*}
Note that $E \subset \Omega(v)$ and $(x_v, t_v) \in \partial E$.
Let $m_W$ denote the normal velocity of the boundary of $E$ at $(x_v, t_v)$.
Since $\Omega(v)$ is nondecreasing, we must have $m_W \geq 0$.
But we can also estimate $m_Z - \delta \geq m_W$ and therefore
\begin{align*}
m_Z - \delta \geq m_W \geq 0.
\end{align*}
We conclude in particular that $t_u > \hat t \geq t_v$.

\subsubsection{Gradients and velocities}

Since $\set{g_2 = \infty} \subset \cl\Omega(v)$ and $(x_v, t_v) \in \Omega^c(v)$, we must have $(x_v, t_v) \in \cl{\set{g_2 < \infty}}$.
Since $(x_u, t_u), (x_v, t_v) \in \cl\Xi_r(\hat x, \hat t) \subset B_{\hat r}(\hat x, \hat t)$, we can estimate
\begin{align*}
g_1(x_u, t_u) \leq \sup_{\cl\Xi_r(\hat x, \hat t)} g_1 \leq \sup_{\cl B_{\hat r}(\hat x, \hat t)} g_1 \leq \inf_{\cl B_{\hat r}(\hat x, \hat t)} g_2 < \infty.
\end{align*}

Let $\nu$ be the unit outer normal to $\set{x: (x, \hat t) \in \Xi_r(x_u, t_u)}$.
We can define the ``weak gradients''
\begin{align*}
\alpha := \limsup_{h\to 0+} \frac{Z(\hat x - h \nu, \hat t)}h, \qquad \beta := \liminf_{h\to 0+} \frac{W(\hat x - h \nu, \hat t)}h.
\end{align*}
Since $\hat x$ is a regular point of the boundary $\partial U$, weak Hopf's lemma implies $\alpha \leq \beta$, $\alpha < \infty$ and $\beta > 0$.

As $t_u > 0$, we have enough space to put a barrier above $u$ as in \cite{KP13} in a neighborhood of $(x_u, t_u)$ and prove that
\begin{align*}
m_Z \leq g_1(x_u, t_u) \alpha < \infty.
\end{align*}
Therefore $m_W < \infty$.
In particular, $t_v > \hat t - r + \delta \hat t$.
Therefore we have enough space to put a barrier under $v$ as in \cite{KP13} in a neighborhood of $(x_v, t_v)$ and prove that
\begin{align*}
\infty > m_Z - \delta \geq m_W \geq g_2(x_v, t_v) \beta.
\end{align*}
Note that a subbarrier does not need $g_2 < \infty$ in the complement of its support
Definition~\ref{def:barrier}.
In particular, $g_2(x_v, t_v) < \infty$.
Putting all together, we have
\begin{align*}
m_Z \leq g_1(x_u, t_u) \alpha \leq g_2(x_v, t_v) \beta \leq m_Z -\delta,
\end{align*}
a contradiction.

\subsection{Well-posedness of \eqref{FB}}

We have the following existence and uniqueness result for \eqref{FB}.

\begin{theorem}[Well-posedness]
\label{th:hs-well-posedness}
Suppose that $\Omega_0 \subset \Rn$ is a bounded open set with a $C^{1,1}$ boundary.
Moreover, let $\rho^E_0 \in C(\Rn)$ be a function such that $0 \leq \rho^E_0 < 1$,
$\lim_{\abs x\to\infty} \rho^E_0(x) = 0$.
Then there exists a unique viscosity solution $u$ of \eqref{FB}
with initial support $\Omega_0$ and initial density $\rho^E_0$
in the sense that $u$ is a viscosity solution of \eqref{hsg}
in $Q = \Rn \times (0,\infty)$
with
\begin{align}
\label{g-wellposedness}
g(x,t) :=
\frac 1{1 - \min\set{\rho^E_0(x) e^{t G(0)}, 1}},
\end{align}
where $g = \infty$ if the denominator is $0$,
and $u$ satisfies the initial condition as
\begin{align*}
\set{x : u^{*,Q}(x, 0) > 0} = \set{x : u_{*,Q}(x, 0) > 0} = \Omega_0.
\end{align*}
The solution is unique in the sense that if $u$ and $v$ are two viscosity solutions of
\eqref{hsg} with the same initial data, then
\begin{align}
\label{uniquenesssolutions}
u^{*,Q} = v^{*,Q}, \qquad u_{*,Q} = v_{*,Q}.
\end{align}
\end{theorem}

In the proof of the uniqueness in this theorem
we also obtain the following version of the comparison principle.

\begin{theorem}
\label{th:comparisonhsg}
Let $\Omega_0$ and $\rho^E_0$ be as in Theorem~\ref{th:hs-well-posedness}.
Suppose that $u$ is a viscosity subsolution and $v$ is a viscosity supersolution of \eqref{hsg}
in $Q= \R^n \times (0, \infty)$ with $g$ as in \eqref{g-wellposedness},
with the initial data
\begin{align*}
\set{x : u^{*,Q}(x, 0) > 0} \subset \Omega_0 \subset \set{x : v_{*,Q}(x, 0) > 0}.
\end{align*}
Then
\begin{align*}
u^{*,Q} \leq v^{*,Q}, \qquad u_{*,Q} \leq v_{*,Q}.
\end{align*}
\end{theorem}

We now proceed with the proof of the well-posedness theorem.
Let $u$ and $v$ be two solutions of \eqref{FB} on $Q = \Rn \times (0, \infty)$ with the given initial data.
We want to prove that they must be equal in the sense of \eqref{uniquenesssolutions}.

The basic idea is to perturb one of the solutions to create a strictly ordered pair and then apply the comparison principle.
To apply Theorem~\ref{th:comparison}, for $\alpha > 1$ and $\sigma > 0$ we consider the rescaled shifted function
\begin{align*}
w(x,t) = v(x, \alpha t + \sigma).
\end{align*}
Clearly $w \in \supers(g_2; Q)$, where
\begin{align*}
g_2(x) = \alpha g(x, \alpha t + \sigma).
\end{align*}

We want to show that we can find $\hat r > 0$ such that the assumptions of the comparison principle Theorem~\ref{th:comparison} are satisfied.

\begin{lemma}
\label{le:g-shift-big}
Suppose that $g$ satisfies the assumptions \eqref{g-assumption}, $g$ is nondecreasing in time, and $\set{g = \infty}$ is the epigraph of a function $\tau :\Rn \to \R$ such that $\tau$ is continuous at every point in $\set{\tau < \infty}$.
Then for every compact set $E \subset \Rn\times [0,\infty)$, $\alpha > 1$ and $\sigma > 0$ there exists $r > 0$ such that $\alpha g(x, \alpha t + s) \geq g(y,s)$ whenever $(x,t), (y,s) \in E$ and $\abs{(x,t) - (y,s)} \leq r$.
\end{lemma}

\begin{proof}
Let us set $f(x,t) := \alpha g(x, \alpha t + \sigma)$.

\textbf{1.} We first show that we can find $K > 0$ such that
\begin{align*}
\delta_1 := \dist(\set{g \geq K}, \set{f < \infty} \cap E) > 0.
\end{align*}
Indeed, suppose that $\delta_1 = 0$ for any $k \in \N$.
Thus we can find sequences $(x_k, t_k) \in \set{g \geq k}$, $(y_k, s_k) \in \set{f < \infty} \cap E$ such that $\abs{(x_k, t_k) - (y_k, s_k)} < \frac1k$.
By compactness of $E$ we can assume that up to a subsequence $(x_k, t_k) \to (\hat x, \hat t)$ and $(y_k, s_k) \to (\hat x, \hat t)$ for some $(\hat x, \hat t) \in E$.
In particular, $0 \leq \hat t < \infty$.
Since we have $g(\hat x, \hat t) \geq \liminf_{k\to\infty} g(x_k, t_k) = \infty$ by \eqref{g-assumption}, we deduce $\tau(\hat t) \leq \hat t$.
Furthermore, as $\alpha s_k + \sigma < \tau(y_k)$, continuity of $\tau$ yields
\begin{align*}
\alpha \hat t + \sigma \leq \tau(\hat x) \leq \hat t,
\end{align*}
a contradiction.
Therefore we can choose $K > 0$ such that $\delta_1 > 0$.

\textbf{2.} Let $\delta_2 := \dist(\set{f = \infty}, \set{f \leq K} \cap E)$,
and we observe that $\delta_2 > 0$ due to \eqref{g-assumption} and the compactness of $E$.

Since the set $Q :=\set{f \leq K} \cap E \subset \set{g < K}$ is compact, we can find a modulus of continuity $\omega$ of both $g$ on this set,
and $m : = \min_Q \min\set{f, g} > 0$.
Let us find $\rho > 0$ such that $\omega(\rho) \leq (\alpha - 1) m$.
We set $r := \frac12 \min \set{\delta_1, \delta_2, \rho}$.

\textbf{3.} Choose now $(x,t), (y,s) \in E$ with $\abs{(x,t) - (y,s)} \leq r$.
We now prove that $f(x,t) \geq g(y,s)$.

\begin{itemize}
\item If $f(x,t) = \infty$ then the conclusion is trivial.
\item If $K \leq f(x,t) < \infty$ then $g(y,s) < K$ and hence $f(x,t) \geq g(y,s)$.
\item If $K \leq g(y,s)$ then $f(x,t) = \infty$, and again the conclusion is trivial.
\item Finally, if neither of the above is satisfied, we must have $f(x,t) \leq K$ and $g(x,t) \leq K$.
Therefore we can estimate using the monotonicity in time and continuity
\begin{align*}
f(x,t) &= \alpha g(x, \alpha t + \sigma) \geq \alpha g(x,t)\\
&= (\alpha - 1) g(x,t) + g(x,t) \geq (\alpha - 1) m + g(y,s) - \omega(r)\\
&\geq g(y,s).
\end{align*}
\end{itemize}
This finishes the proof.
\end{proof}

\begin{proof}[Proof of Theorem~\ref{th:hs-well-posedness}]
\textbf{Uniqueness.}
Let us first prove uniqueness.
Suppose that $u$ and $v$ are two viscosity solutions satisfying the inital condition.
For simplicity, in the following we write $u$ instead of $u^{*,Q}$,
and $v$ instead of $v_{*,Q}$.

1. If $u$ is a viscosity solution with initial condition $\Omega_0$, a bounded set,
we can compare it with a large radially symmetric superbarrier
\begin{align}
\label{vt-barrier}
W_T = \frac{G(0)}{n} (R^2 e^{16G(0)t /n}  - \abs x^2).
\end{align}
Indeed, since $\rho^E_0 \to 0$ as $\abs x \to \infty$,
we can for any $T > 0$ find $R$ sufficiently large such that
$\rho(x,t) < \frac 12$ for $\abs x \geq R$, $0 \leq t \leq T$
and $\Omega_0 \subset B_R(0)$.
Then $W_T$ is a superbarrier for $0 \leq t \leq T$ since $-\Delta W_T = 2 G(0) > G(0) \geq G(W_T)$
in $\{W_T > 0\}$, while
\begin{align*}
\frac{\partial_t W_T}{|DW_T|^2} = 4 > 2 \geq \frac 1{1- \rho} = g(x,t) \qquad \text{on } \partial \{W_T > 0\}.
\end{align*}
The comparison with this superbarrier yields that $\Omega(u; \Rn \times [0,T]) \subset B_{R
e^{16G(0) T/n}}(0) \times [0, T]$.
Let us therefore define
$Q_T = B_{2(Re^{16 G(0) T/n})}(0) \times [0,T]$.

2. We apply the comparison principle on $Q_T$.
Since $\Omega_0$ has interior ball condition,
by comparison with radial subbarriers we can prove that
$\Omega_0 \subset\subset \Omega_t(v)$ for $t > 0$.
To see this, consider the function
\begin{align*}
w(x,t) = \alpha ((c t + r)^2 - |x - x_0|^2).
\end{align*}
For given $0 < r < 1$, we can first choose $0 < \alpha \ll 1$ such that $G(4\alpha) >
2n\alpha$ and then choose $0 < c \ll 1$ so that $c(c + r)/ (2\alpha r^2) < 1$.
Then for $0 \leq t \leq 1$ we have $w \leq 4 \alpha$ and therefore $G(w) \geq G(4\alpha) > \Delta w$.
Moreover,
\begin{align*}
\frac{w_t}{|Dw|^2} = \frac{2 \alpha c(c t + r)}{4 \alpha^2 |x - x_0|^2} \leq \frac{c(ct + r)}{2
\alpha r^2} < 1 \leq g \qquad \text{on } \partial \{w > 0\}.
\end{align*}
We see that $w$ is a subbarrier for $0 \leq t \leq 1$.
We conclude that if $\Omega_0$ has a interior condition with radius $r > 0$, the free boundary of a solution must
expand initially with velocity at least $c > 0$ given above.

3. Let us fix $\sigma > 0$.
We can find an open set $U \subset \R^n$ with smooth boundary such that
$\Omega_0 \subset\subset U \subset\subset \Omega_\sigma(v)$.
$\Omega_t(u)$ cannot jump outside of $\Omega_0$ by the definition of
a viscosity solution
and therefore $\cl{\Omega_t(u)} \subset U$ for all $t > 0$ sufficiently small.
By the strong maximum principle for the elliptic problem,
we obtain that the solution of the elliptic problem on $U$ with data zero on
$\partial U$ is strictly smaller then the solution of the elliptic problem on $\Omega_\sigma(v)$.
Since $x \mapsto u(x,t)$ is a subsolution of the elliptic problem on $\Rn$ for any $t > 0$,
and $x \mapsto v(x,\sigma)$ is a supersolution of the elliptic problem on $\Omega_\sigma(v)$,
we conclude that $u(\cdot, 0) < v(\cdot, \sigma)$ on $\Omega_\sigma(v)$.

Let us define $w(x,t) = v(x, (1+\sigma)t + \sigma)$ for some $\sigma > 0$.
By the reasoning above, $u \prec_{Q_T} v$ on $\partial_P Q_T$.
Lemma~\ref{le:g-shift-big} implies that the functions $g_1 = g$ and $g_2(x,t) = (1+\sigma) g(x, (1+\sigma)t + \sigma)$
satisfy
the assumptions of Theorem~\ref{th:comparison} on $Q_T$.
Therefore $u \leq w$ on $\Rn \times [0,T]$.

Now we send $\sigma \to 0+$
and recover
\begin{align*}
u_* \leq (u^*)_* \leq v_*.
\end{align*}
By shifting $u$ instead of $v$, that is, considering $u(x, (1 + \sigma)^{-1}(t - \sigma))$
and then sending $\sigma \to0+$,
we also obtain
\begin{align*}
u^* \leq (v_*)^* \leq v^*.
\end{align*}

By repeating the same argument with $u$ and $v$ interchanged, we obtain the uniqueness of solutions:
\begin{align*}
u^* = v^*, \qquad u_* = v_*.
\end{align*}

\textbf{Existence.}
Existence follows from standard Perron-Ishii's method.
We first construct appropriate barriers.

1. Let $Z_\rho$ for $\rho \geq 0$ be the unique solution of the elliptic problem in
$\Omega_0 + B_\rho(0)$ with boundary value zero,
and zero outside of $\Omega_0 +B_\rho$,
where $B_0(0) = \set{0}$.
Since $\Omega_0 \in C^{1,1}$, we see that $\Omega_0 + B_\rho(0) \in C^{1,1}$ for small
$\rho > 0$ and therefore such $Z_\rho$ exists.
Clearly $U(x,t) = Z_0(x)$ is viscosity subsolution of \eqref{hsg}
in $\Rn \times (0,\infty)$.

On the other hand, let us define
\begin{align*}
V(x,t) =
\begin{cases}
Z_{at}(x) & 0 \leq t \leq \eta,\\
W_1(x,t) & \eta < t \leq 1,\\
W_k(x,t) & k -1 < t \leq k, \text{ iteratively $k = 2, 3, \ldots$},
\end{cases}
\end{align*}
where $W_k$ (with $k = T$) was defined in \eqref{vt-barrier}.
Since $\rho^E_0 < 1$ on $\partial \Omega_0$,
$g$ as defined in \eqref{g-wellposedness} is finite in a neighborhood of
$\partial \Omega_0 \times \set{0}$.
Therefore by continuity, we can find $\eta > 0$ sufficiently small and $a > 0$ large enough
so that $V_T$ is a viscosity supersolution.

Note that by continuity of $U$ and $V$ for all $t \geq 0$ small, we have
\begin{align*}
U^*(x, 0) = U_*(x,0) = V^*(x,0) = V_*(x,0) = Z_0(x).
\end{align*}

2. Let now $u$ be the supremum of viscosity subsolutions $w$ with initial data $w^{*,Q}(x,0) = U(x,0)$.
Since $U$ belongs to this class, we see that $u$ is well-defined and $u \geq U$.
Moreover, the comparison principle, with the perturbation above in the proof of uniqueness,
yields
\begin{align*}
U^* \leq u^* \leq V^*, \qquad U_* \leq u \leq V_*.
\end{align*}
In particular, $u$ has the correct initial data.
We only need to show that it is a solution.
We use Definition~\ref{def:visc-barrier}.
Let us show that u is a subsolution.
If not, there exists a parabolic neighborhood and a superbarrier which $u$ crosses,
even though they are strictly ordered on the parabolic boundary.
In this case, we can perturb the barrier at the crossing point (making it smaller)
and deduce that one of the subsolutions must cross the
perturbed barrier, leading to a contradiction.

Similarly, to show that it is a supersolution, we suppose that $u$ crosses a subbarrier.
if this happens, we can perturb the subbarrier, making it larger, and since the
perturbed subbarrier is a viscosity subsolution, this makes $u$ larger,
contradicting the maximality of $u$.
We therefore only need to check that $\set{g = \infty} \cap Q \subset \cl\Omega(u_{*,Q}; Q)$.
But by our assumption on $\rho^E_0$ we have $\set{g = \infty} = \cl{\interior \set{g = \infty}}$.
Suppose that $\cl B_\rho(\xi) \times \set{\tau} \in \interior{g = \infty}$ for some $(\xi, \tau)$ and
$\rho > 0$.
Let $z$ be the solution of the elliptic problem on $B_\rho(\xi)$,
and $0$ outside of $B_\rho(\xi)$.
Then
\begin{align*}
Z(x,t) =
\begin{cases}
0 & t < \tau,\\
z(x) & t \geq \tau,
\end{cases}
\end{align*}
is a viscosity subsolution.
In particular, $u_* > 0$ in $B_\rho(\xi) \times \set{t > \tau}$.
From this we conclude that $\interior\set{g = \infty} \cap Q \subset \Omega(u_{*,Q}; Q)$,
and thus $\set{g=\infty} \cap Q \subset \cl\Omega(u_{*,Q}; Q)$.

We have proved that $u$ is the unique solution of \eqref{hsg} with $g$ of the form
\eqref{g-wellposedness}
and initial support $\Omega_0$.
\end{proof}

\begin{corollary}\label{cor:00}
Suppose $\rho^E_0$ is Lipschitz. Then $\partial\{p>0\}$ has Lebesgue measure zero in $\R^n\times [0,\infty)$.
\end{corollary}
\begin{proof}
We will show the following density estimate, which is sufficient to conclude: For any $T>0$, there exists $k=k(T)>0$ such that for any space-time ball $B_{n+1}$ with radius $r(T+2)$ centered at $(x_0,t_0)\in \partial\{p>0\}\cap\{t\leq T\}$, there is a space-time ball $\tilde{B}_{n+1}$ of radius $kr>0$  which lies in both $\{p>0\}$ and in $B_{n+1}$. 

To show this, let us first prove the ordering
\begin{equation}\label{density}
p_1(x,t):=\sup_{|x-y|\leq kr} p(x,t) \leq p_2(x,t):=p(x, (1+r)t+r),
\end{equation}
which would hold for an appropriate choice of $k$ if $\rho^E_0$ is Lipschitz. To see this, first note that the order holds at $t=0$, due to the step 2 in the proof of Theorem ~\ref{th:hs-well-posedness}.

It is straightforward to check that  $p_1$ is a viscosity subsolution of \eqref{FB} with modified normal velocity $ V = |Dp_1|g_1 $ with
$$
g_1(x,t) = \frac{1}{1-\min[1,\rho_E+e^{G(0)t}\omega(kr)]},
$$ where $\omega$ is the continuity mode of $\rho^E_0$, and $p_2$ is a viscosity supersolution of \eqref{FB} with normal velocity $V= |Dp_2|g_2$, where  
$$
g_2(x,t) = \frac{1+r}{1-\min[1,e^{G(0)r(1+t)}\rho^E)]}.
$$
 Now suppose $\rho^E_0$ is Lipschitz so that $\omega(s) \leq Ls$ for some constant $L>0$. Then note that if we choose $k=\frac{1}{2L}e^{-G(0)[T+1]}$, then we have  $g_1\leq g_2$ for $0<r<1$. Therefore the comparison principle for \eqref{FB} yields \eqref{density}. 
 
 Now to check our original claim, suppose $(x_0,t_0)\in \partial\{p>0\}\cap \{t\leq T\}$. Let $p_1$ as given in \eqref{density}, then the spatial ball $\tilde{B}$ of radius $kr$ and center $x_0$ lies in the positive set of $p_1$. Due to \eqref{density}, $\tilde{B}$ also lies in the positive set of $p$ at time $t_1:=(1+r)t_0+r$. Due to the monotone increasing nature of $p$, we then end up with a space-time cylinder $B_{kr}(x_0)\times [t_1+ t_1+kr]$ lying in the positive set of $p$. Since $t_1 \leq t_0 + r(T+1)$, we can conclude that our density estimate holds.

\end{proof}

\section{Convergence in Local Radial setting}
\label{se:convergence radial}

Here we will introduce the notion of radial solutions and give the convergence proof. To make local perturbations of general barriers to make first-order approximations in space and time, we need to consider radial barriers with fixed boundaries.

\begin{definition}
The definition will follow the notion via barriers in Definition~\ref{def:visc-barrier}, which considers $\rho$ outside of the tumor region $\{ p > 0 \}$
as given a priori by $\rho^E(x,t) = \rho^E_0(x) e^{tG(0)}$.
\end{definition}

\begin{definition}[Radial solutions]
\label{de:radial solutions}
$(\phi,\rho_{\phi}^E) $ is a radial, classical solution of \eqref{FB}
in  the cylindrical domain $\{|x-x_0| \leq R\}\times [t_1,t_0]$ or $\{|x-x_0| \geq R\} \times [t_1,t_0]$  if
$$
\rho_{\phi}^E (x,t)= \rho_\phi^E(x,t_1)e^{(t-t_1)G(0)}, \qquad \rho_\phi^E(\cdot, t_1) \in C^2(\R^n)
$$
and

\begin{itemize}
\item[(a)]$\phi(\cdot,t)$  is radial with respect to $x_0$ and is smooth in its positive phase;\\
\item[(b)] $\phi$ solves \eqref{FB} in the classical sense with the free boundary motion law $V=\dfrac{|D\phi|}{1-\rho_{\phi}^E}$;\\
\item[(c)] $\phi(\cdot,t) >0$ in $|x-x_0|=R$ for $t_1\leq t\leq t_0$;\\
\item[(d)] $\rho^E_{\phi}<1$ outside of $\{\phi>0\}$.
\end{itemize}

\end{definition}

\begin{lemma}
\label{le:limit-uniqueness}
The pair $(\chi_{\set{\phi > 0}} + \chi_{\set{\phi=0}} \rho_\phi^E, \phi)$ is the unique pair of functions $(\rho, p)$ in $L^\infty(Q)$,
$\rho \in C([0,\infty]; L^1(\Rn))$, $p \in P_\infty(\rho)$, satisfying
\begin{align}
\label{lim-distr}
\begin{aligned}
\partial_t \rho = \Delta p + \rho G(p) \quad \text{in $\mathcal D'(Q)$},
\quad \rho(0) = \rho_\phi^E(0) \quad \text{in $L^1(\Rn)$},\\
p = \phi \quad \text{on $\partial\Omega$ in the sense of trace in $W^{1,2}(\Rn)$ for a.e. $t > 0$}
\end{aligned}
\end{align}
such that
\begin{enumerate}
\item $\rho, p \in L^\infty((0,T); L^1(\Rn))$;
\item $\rho(t)$ is uniformly compactly supported in $t \in [0,T]$;
\item $\abs{\nabla p} \in L^2(Q_T)$;
\item $\partial_t p \in \mathcal M(Q_T)$, $\partial_t \rho\in \mathcal M(Q_T)$.
\end{enumerate}
Here $P_\infty$ is the Hele-Shaw monotone graph
\begin{align*}
P_\infty(\rho) =
\begin{cases}
\set0 & 0 \leq \rho < 1,\\
[0, \infty) & \rho = 1.
\end{cases}
\end{align*}
\end{lemma}

\begin{proof}
Let us first prove the uniqueness of the solutions of \eqref{lim-distr}.
The statement is analogous to \cite[Theorem~2.4]{PQV},
with the extra boundary condition for $p$.

To apply the Hilbert's duality method that is outlined in \cite[Section~3]{PQV},
we need any two solutions $(\rho_i, p_i)$ to satisfy
\begin{align}
\label{hilbert}
\int_{Q_T} (\rho_1 - \rho_2) \psi_t + (p_1 - p_2) \Delta \psi + (\rho_1 G(p_1) - \rho_2 G(p_2))
\psi \dx \dt = 0
\end{align}
for all $\psi \in C^\infty(\cl{Q_T})$ with boundary data zero on $\partial \Omega$ and
at $t = T$.
For $\psi \in C^\infty_c(Q_T)$ this follows
\eqref{lim-distr}.
Then this can be extended to include $\psi$ nonzero at $t = 0$
as in \cite{Vazquez}.
To extend this to all $\psi$ whose support touches the boundary,
we need to approximate $\Delta \psi$ by $\Delta\varphi$, $\varphi \in C^\infty_c(Q_T)$
in the correct norm (at least $L^1$ since $p_1 -p_2 \in L^\infty$.)
However, this is not possible since $\nabla \psi \neq 0$ on the boundary in general.
We therefore use the fact that $\nabla p \in L^1(0, T; H^1(\Omega))$,
and show first
\begin{align*}
\int_{Q_T} (\rho_1 - \rho_2) \psi_t - \nabla(p_1 - p_2) \cdot \nabla \psi
+ (\rho_1 G(p_1) - \rho_2 G(p_2))
\psi \dx \dt = 0,
\end{align*}
by approximation
and then integrate the second term by parts in space and use that $p_1 = p_2$ on
$\partial \Omega$.
Then we just follow \cite[Section~3]{PQV} since the rest does not see the boundary values.

To finish the proof, we have to show that $(\rho_\phi^E, \phi)$ satisfies \eqref{lim-distr}.
Let us set $p = \phi$ and $\rho = \chi_{\set{\phi > 0}} + \chi_{\set{\phi = 0}} \rho_\phi^E$.
We see that $p \in P^\infty(\rho)$, $(\rho, p)$ has all the regularity required by the assumptions on $(\rho_\phi^E, \phi)$,
and has the correct initial and boundary data.
We therefore only need to show that it satisfies \eqref{lim-distr} in the sense of distributions.
Let $\varphi \in C^\infty_c(Q_T)$ be a test function. We will verify that
\begin{align*}
\int_{Q_T} \rho \varphi_t + p \Delta \varphi + \rho G(p) \varphi \dx \dt = 0.
\end{align*}
Since the boundary $\partial \set{p > 0}$ is assumed to be smooth, its unit outer normal is $\frac 1{\sqrt{1 + V^2}} \pth{-\frac{\nabla p}{\abs{\nabla p}}, -V}$ where $V$ is the normal velocity of $\partial\set{p>0}$ at the given boundary point. Therefore it follows that
\begin{align*}
\int_{Q_T} \rho \varphi_t &= -\int_{\set{p = 0}} \rho_t \varphi - \int_{\partial\set{p > 0}} (1 - \rho) \varphi \frac V{\sqrt{1 + V^2}} \dS,\\
\int_{Q_T} p \Delta \varphi &= \int_{\set{p > 0}} \varphi \Delta p + \int_{\partial\set{p > 0}} \abs{\nabla p} \varphi \frac 1{\sqrt{1 + V^2}} \dS,\\
\int_{Q_T} \rho G(p) \varphi &= \int_{\set{p > 0}} G(p) \varphi + \int_{\set{p = 0}} \rho G(0) \varphi,\\
\end{align*}
We see that the sum of these terms gives zero.
\end{proof}

\medskip

To avoid an initial layer in the limit $m\to\infty$, we need to match the initial data for the
$m$-problems.  For given radial solution $(\phi, \rho^E_\phi)$ we therefore define the initial data
for $\rho_m$ by first finding $\delta > 0$ so that $\rho^E_\phi(x, 0) < 1 - \delta$ on
$\{\phi(\cdot, t) = 0\}$ and then setting
\begin{align}
\label{rhom0}
\rho_{0,m}(x) = \max \bra{P_m^{-1}(\phi(x, 0)),
\min(1 - \delta, \rho^E_\phi(x, 0))}.
\end{align}
Such $\delta$ can be found due to the assumption (d) in Definition~\ref{de:radial solutions} and
the continuity of $\rho^E_\phi$.  Note that with above choice of $\rho_{0,m}$ we have $\rho_{0,m}
\to \chi_{\set{\phi > 0}} + \chi_{\set{\phi = 0}} \rho^E_\phi$ in $L^1$ and $\norm{\nabla
\rho_{0,m}}_{L^1} \leq C$, and $p_{0,m} = P_m(\rho_{0,m}) \to \phi$ uniformly. Moreover $\rho_m$ is
nondecreasing in time when $m$ is large enough according to the following lemma.

\begin{lemma}
\label{le:rho matched increase}
Solution $\rho_m$ of \eqref{pme} with initial data $\rho_{0,m}$ given in \eqref{rhom0} is
nondecreasing in time
for $m$ sufficiently large depending only on $\norm{\Delta \rho^E_0}_\infty$.
\end{lemma}

\begin{proof}
Note that due to the comparison principle, we only need to show that $\rho_m$ is nondecreasing at
the initial time.
By the continuity of $\rho^E_\phi$ we can choose a compact set $K \subset \{ \phi > 0 \}$ such that
$\rho^E_\phi < 1 - \delta$ on $K^c$.
We can find $m_0$ so that $ \phi > P_m(1 - \delta)$ on $K$ for
$m \geq m_0$.
In particular, when $m \geq m_0$, we see that $P_m(\rho_{0,m}(x)) = \phi(x, 0)$ or $\rho_{0,m}(x) =
\rho^E_\phi(x, 0) < 1 - \delta$.

In terms of the pressure variable $p_m = P_m(\rho_m)$, $\rho_m$ satisfies the equation
\begin{align*}
\partial_t p_m = (m-1) p_m (\Delta p_m + G(p_m)) + |\nabla p_m|^2.
\end{align*}
We know that $\phi$ satisfies $\Delta\phi + G(\phi) = 0$ at $t = 0$ and therefore $\partial_t p_m
\geq 0$ at the points where $P_m(\rho_{0,m}) = \phi(x, 0)$.

On the other hand, if $\rho_{0,m} = \rho^E_\phi < 1 - \delta$ then at such point $(x,0)$ we have
\begin{align}
\label{press gradient bound}
\Delta p_{0,m} = m(m-2) \rho_{0,m}^{m-3} |\nabla \rho_{0,m}|^2 + m \rho_{0,m}^{m-2} \Delta
\rho_{0,m} \geq m \rho_{0,m}^{m-2} \Delta \rho_{0,m}.
\end{align}
Since $p_{0,m} \leq P_m(1 - \delta) \ll 1$ we have $G(p_{0,m}) \geq \frac{G(0)}2$
for sufficiently large $m$.
Since $\norm{\Delta \rho^E_\phi}_\infty \leq C$,
we have from \eqref{press gradient bound}
\begin{align*}
\Delta p_{0,m} + G(p_{0,m}) \geq - m (1-\delta)^{m-2} \norm{\Delta \rho_{0,m}}_\infty + G(0)/2 > 0
\end{align*}
for sufficiently large $m$ uniformly in $x$.
We conclude that $\partial_t p_m \geq 0$.
Therefore $\rho_m$ is nondecreasing for sufficiently large $m$.
\end{proof}

\begin{theorem}\label{thm:perthame:ext}
For a given radial solution $\phi$
on $\set{\abs{x - x_0} < R} \times (t_1, t_0)$ or $\set{\abs{x - x_0} > R} \times (t_1, t_0)$,
the corresponding solutions $p_m, \rho_m$ of \eqref{pme}
on the same domain
with initial data
$\rho_m(\cdot, t_1) = \rho_m^0$ at $t = 0$
and boundary data $p_m = \phi$ on $|x|=R$  satisfy the following:
$p_m$ uniformly converges to
$\phi$, and $\rho_m$ uniformly converges to $\rho_\phi^*$ away from the support of $\phi$.
\end{theorem}

\begin{proof}
We consider the case of exterior domain.
We will for simplicity assume that $x_0 = 0$, $t_1 = 0$ and $t_0 = \infty$.
Let $T > 0$.
Let $\Omega = \set{R < \abs x < r}$ for $r \gg R$ large enough so that it contains the support of solutions $p_m$ for $t \leq T$.
We shall assume that $\phi$ is smooth up to the boundary $\partial \Omega$
and $\phi_t \geq 0$ on $\partial \Omega \times [0, \infty)$.
Then $\phi_t \geq 0$ in $\set{\phi > 0}$.
We set
\begin{align*}
\kappa = \min \set{\phi(x,t): \abs x = R, t \geq 0} > 0.
\end{align*}

Let us consider the solution $\rho_m$, $p_m = P_m(\rho_m)$, of the porous medium equation \eqref{pme} on
\begin{align*}
Q_T = \Omega \times (0,T),
\end{align*}
with initial data $\rho_m^0$ at $t = 0$ and boundary data $p_m = \phi$ for $\abs x = R$.

\textbf{Estimates.}
Since $\phi$ solves $-\Delta \Phi = G(\Phi)$ in $\Omega \cap \set{\psi > 0}$ for every $t$,
we see that $\partial_t p_m(\cdot, 0) \geq 0$.
Indeed, recall the initial data from \eqref{rhom0}.
For $x$ such that $\frac {m-1}m (\rho_m^0(x))^{m-1} = \phi(x)$, we have
\begin{align*}
\partial_t p_m = \abs{\nabla p_m}^2 \geq 0 \qquad \text{at $t = 0$}.
\end{align*}
On the other hand, if $\rho_m^0(x) = \rho_\phi^E(x, 0)$, we conclude that $\partial_t \rho_m \geq 0$ at $t =0$
for sufficiently large $m$ by the regularity $\rho_\phi^E \in C^2$ and the fact that $\rho_m^0 < 1$.
The transition between these two regimes is a convex corner (maximum of two nondecreasing initial data).
Therefore $\partial_t p_m \geq 0$ by the comparison principle.

By putting a subsolution under $p_m$, we can find $R_{1/2} > R$ such that $p_m(\cdot, t) \geq \kappa/2$ on $\Omega_{1/2} = \set{x: R \leq \abs x \leq R_{1/2}}$.

We first derive the uniform $C^{1,\alpha}$ and $C^{2,\alpha}$ estimates for $p_m$ on $\Omega_{1/2}$.
Let us rescale in time. Note that $\tilde{p}_m(x,t):= p_m(x,\frac{t}{m-1})$ solves the equation
$$
\tilde{p}_t = \tilde{p}\Delta \tilde{p} + \frac{1}{m-1}|\nabla \tilde{p}|^2 + \tilde{p}G(\tilde{p}).
$$
Since $\tilde{p}$ is uniformly away from zero in $\Omega_{1/2} \times [0, (m-1) T]$ and uniformly bounded from above,
this is a uniformly parabolic, quasilinear equation in the set considered above. Now uniform $C^{1,\alpha}$ estimate up to the boundary for $\tilde{p}$, where the $C^{1,\alpha}$ norm is
only depending on the boundary data of $\tilde{p}$ as well as the initial data; see Theorem 4.7 and Theorem 5.3 in \cite{Lieberman86}.
We also have uniform $C^{2,\alpha}$ interior estimates up to the initial boundary.
In terms of $p_m$ we lose the estimate in time,
but we still have the estimate in space.
Namely, for sufficiently small $\e > 0$
there exists a constant $C_T > 0$, independent of $m$, such that
\begin{align*}
\norm{p_m(\cdot, t)}_{C^{1,\alpha}(\cl{\Omega_{1/2})}} + \norm{p_m(\cdot, t)}_{C^{2, \alpha}(\set{R + \e/2 \leq \abs x \leq R+ 2\e})} \leq C_T
\qquad \text{for every $0 \leq t \leq T$.}
\end{align*}

This yields the bound
\begin{align}
\label{pm-estimate}
\abs{D^2 p_m} + \abs{D p_m} \leq C,
\quad \text{on }
\set{(x,t):\abs x = R + \e,\ t \geq 0}.
\end{align}

Since the set $\set{x: \abs x = R}$ is smooth, we can easily create barriers $\phi_1$, $\phi_2$ at the boundary that coincide with $\phi$ on the boundary
and $\phi_1 \leq \phi_2$.
Moreover, $\phi_1$ is a subsolution and $\phi_2$ is a supersolution of
\begin{align*}
p_t = (m-1) p \Delta p + \abs{\nabla p}^2 + (m-1) p G(p),
\end{align*}
We conclude that
\begin{align*}
\phi_1 \leq p_m \leq \phi_2 \qquad \text{in a neighborhood of $\set{\abs x = R}$.}
\end{align*}
This will imply that the limit of $p_m$ will have the correct boundary data.

\textbf{Uniqueness}
We shall prove that $p_m$ and $\rho_m$ converge to the unique solution of the problem
in Lemma~\ref{le:limit-uniqueness}.

The main problem with fixed boundary data arises in the semiconvexity estimate for $p_m$,
a variant of the Aronson-Benilan estimate.
Since the proof relies on the maximum principle for $\Delta p_m$,
we need to handle the boundary value of this function.
To accomplish this, we use the estimate \eqref{pm-estimate}.

Indeed, \cite{PQV} derive that $w = \Delta p_m + G(p_m)$ is a solution of
\begin{align}
\label{Delta-eq}
w_t \geq (m-1) p_m \Delta w + 2m \nabla p_m \cdot \nabla w + (m-1) w^2 - (m-1) \pth{G(p_m) - p_m G'(p_m)} w.
\end{align}
All the arguments here can be made rigorous as explained in \cite[Section~9.3]{Vazquez}.
Since \\$\min_{p \in [0, p_M]} \pth{G(p) - p G'(p)} > 0$,
$W(t) = -\frac 1{(m-1) t}$ is a subsolution of
\eqref{Delta-eq}.

Since on $\Gamma = \set{(x,t): \abs x = R+\e,\ t \geq 0}$ we have \eqref{pm-estimate}, we get
\begin{equation}\label{AB}
w = \Delta p_m + G(p_m) \geq \Delta p_m \geq - C \qquad \text{on $\Gamma$}
\end{equation}
for some constant $C > 0$, independent of $m$.
Let $T = \sup\set{t > 0: W(t) \leq - C} = \frac C{m-1}$.
Thus $W(t)$ is a subsolution of \eqref{Delta-eq} with boundary data $w(x,t) \geq W(t)$ on $\Gamma \cap \set{t \leq T}$ and therefore $W(t) \leq w(x,t)$ on $\set{0 \leq t \leq T}$.
By a bootstrap argument with a shift $W(t - \tau)$ for arbitrary $\tau > 0$,
we can deduce that $w(x,t) \geq -C$ on $\set{(x,t): \abs x > \e,\ t \geq T}$.

With \eqref{AB}, we can recover all the uniform local $L^1$-estimates on $\partial_t \rho_m$, $\nabla \rho_m$,
$\partial_t p_m$, $\nabla p_m$ from section 2 of \cite{PQV}, including the $L^1$-continuity of $\rho_m(t)$ at $t = 0$.
A standard argument implies that $\rho_m \to \rho_\phi^*$ and $p_m \to \phi$ in $L^1_{\rm loc} (\R^n\times [0,\infty))$ by the uniqueness result (Lemma~\ref{le:limit-uniqueness}).

\textbf{Lipschitz estimate.}
The functions $p_m$ and $\rho_m$ depend only on $r = \abs x$ and $t$.
In spherical coordinates, \eqref{AB} reads
\begin{align*}
p_{rr} + \frac{n-1}r p_r + G(p) \geq \min\pth{- \frac 1{(m-1)t}, -C}.
\end{align*}
We observe that $p_{rr} + \frac{n-1}r p_r = r^{1-n}\frac \partial{\partial r} (r^{n-1} p_r)$.
Therefore, for given fixed $t$ and all $m$ large so that $\frac 1{(m-1)t} < C$ we have for $C_1 = C + G(0)$
\begin{align*}
r^{1-n}\frac \partial{\partial r} (r^{n-1} p_r) \geq - C_1.
\end{align*}
Integration yields
\begin{align*}
r_2^{n-1} p_r(r_2, t) - r_1^{n-1} p_r(r_1, t) \geq - \frac{C_1}n \pth{r_2^n - r_1^n}, \quad r_1 < r_2.
\end{align*}
To get the lower bound on $p_r(r)$, $r > R+\e$, we use interior parabolic estimates
\eqref{pm-estimate}
which yield
$\abs{p_r(R + \e, t)} \leq C$.
Therefore
\begin{align*}
p_r(r, t) \geq -C \pth{\frac{R+\e}r}^{n-1} - \frac {C_1 r}n \pth{1 - \pth{\frac{R+\e}r}^n}, \quad r > R+\e.
\end{align*}
To get the upper bound, we recall that $0 \leq p \leq p_M$.
By the mean value theorem for any $r > R$ there exists $r_2 \in (r, r+1)$ with $\abs{p_r(r_2, t)} \leq p_M$.
Thus
\begin{align*}
p_r(r,t) \leq \pth{\frac{r_2}r}^{n-1} p_M + \frac{C_1r}n \pth{\pth{\frac{r_2}r}^n - 1} \leq \pth{\frac{r+1}r}^{n-1} p_M + \frac{C_1r}n \pth{\pth{\frac{r+1}r}^n - 1}.
\end{align*}

Therefore $p_m$ is locally uniformly Lipschitz in space for every given time $t > 0$ as long as $m \geq C/t + 1$.

\textbf{Uniform convergence of $p_m$ to $\phi$.}

Let us fix $K \subset \Omega$ compact and $T > 0$.
From above we know that $p_m \to \phi$ in $L^1_{\rm loc}(\R^n)$.
We can find a contable set
$\set{t_i}_{i\in \mathbb N} \subset \set{t \geq 0}$ dense in $\set{t \geq 0}$
and a subsequence of $p_m$, still denoted by $p_m$,
such that $p_m(t_i) \to \phi(t_i)$ in $L^1(K)$ for every $t_i$.
We can choose $t_1 = 0$ since $p_m(\cdot, 0) \to \phi(\cdot, 0)$ uniformly
by the choice of $\rho_m^0$ in \eqref{rhom0}.
Due to the uniform Lipschitz bound, by taking a subsequence if necessary, we can assume that
$p_m(\cdot, t_i) \to \phi(\cdot, t)$ uniformly on $K$ for every $t_i$.
Let us choose $\e > 0$.
$\phi$ is uniformly continuous on $K \times [0,T]$
and so there exists $\delta > 0$ such that $\abs{\phi(x, t) - \phi(x, s)} < \e$ for
any $\abs{t - s} < \delta$, $x \in K$.
Find $N \in \mathbb N$ such that $\bigcup_{i=1}^N (t_i - \delta/4, t_i + \delta/4) \supset [0,T + \delta]$
and $M \in \mathbb N$ such that $\norm{p_m(\cdot, t_i) - \phi(\cdot, t_i)}_\infty < \e$ for
all $i= 1, \ldots, N$, $m \geq M$.
Let now $t \in [0,T]$.
We can find $1 \leq i,j \leq N$ such that $t_i \leq t \leq t_j$, $t_j - t_i < \delta$.
Recall that $t \mapsto p_m(x, t)$ is nondecreasing.
Thus for any $x \in K$ and $m \geq M$ we have
\begin{align*}
p_m(x,t) - \phi(x,t) \leq p_m(x,t_j) - \phi(x, t_j) + \phi(x,t_j) - \phi(x,t)
< 2 \e.
\end{align*}
On the other hand
\begin{align*}
p_m(x,t) - \phi(x,t) \geq p_m(x,t_i) - \phi(x, t_i) + \phi(x,t_i) - \phi(x,t) > -2\e.
\end{align*}
We conclude that the subsequence $p_m \to \phi$ uniformly on $K \times [0,T]$.
Since the limit is unique, the whole sequence must converge.

The uniform convergence of $\rho_m$:
We have $\rho_m = \pth{\frac{m-1}m}^{1/(m-1)} p_m^{1/(m-1)}$.
Let $K$ be a compact subset of $\set{p> 0}$.
But the uniform convergence, there exists $\e > 0$ with $p_m \geq \e$ on $K$ for all
$m$ sufficiently large.
Then for every $\delta > 0$ for all $m$ large we have
\begin{align*}
\rho_m \geq \pth{\frac 12 \e}^{1/(m-1)} > 1 - \delta.
\end{align*}
The upper bound follows from the uniform upper bound on $p_m$.
Therefore $\rho_m \to 1$ locally uniformly in $\set{p > 0}$. It remains to show that $\rho_m$ converges to $\rho^*$ locally uniformly away from $\{p>0\}$.

\medskip

Lastly we would like to prove the uniform convergence of $\rho_m$ to $\rho^E_\phi$ outside of $\{p>0\}$.  Due to the definition of our solution,  for each $t_0>0$ there exists $\delta>0$ such that  $\rho^E_\phi <1-\delta$ for some $\delta>0$ outside of $\{p>0\}$ for $0\leq t\leq t_0$. Based on this fact we will argue by iteration over small time intervals as follows:
\medskip

Let us pick $T>0$ small  and choose $x_0$ outside of $\{p>0\}$ so that $\rho^E_\phi(\cdot,T)<1-\delta$ on a ring $r_1<|x|<r_2$ containing $x_0$.  We can pick $r_1, r_2$ such that $\rho_m(\cdot,T)$ uniformly converges to $\rho^E_\phi(\cdot,T)<1-\delta$ at $|x|=r_1$ and $|x|=r_2$. This is possible due to the $L^1$-convergence of $\rho_m \to \rho^E_\phi$. Since $(\rho_m)_t \geq 0$, we conclude that $\rho_m$ stays strictly below $1-\delta$ on $\{|x|=r_i\}\times [0,T]$.

\medskip

We now construct a barrier for $\rho_m$ in $\{r_1<|x|<r_2\}\times [0,T]$ as follows. At $t=0$ we pick a radial, smooth function $\varphi_0(x)$ which has the same value as $\rho_0$ near $x_0$ and has the value $1-\delta$ on the boundary $|x|=r_i$. Now let $\varphi(x,t) = e^{Ct}\varphi_0(x)$ where $C = G(0) + \frac{1}{m}$ or something like this. Since $\rho_m \leq 1-\delta$ this works fine as a subsolution for the $\rho_m$ equation if $T<O(\delta)$, and it follows that $\rho_m$ converges uniformly to $\rho^E_\phi$ at $x_0$. As a consequence we have the uniform convergence of $\rho_m$ to $\rho^E_\phi$ in every compact subset of $\{p>0\}^c$  for $0\leq T\leq O(\delta)$. We now iterate over time to conclude up to $t=t_0$. Since $t_0$ can be chosen arbitrarily large, we conclude.

\end{proof}

\section{Convergence in the general setting}

Based on the Theorem~\ref{thm:perthame:ext},
next we consider general, i.e., non-radial solutions $\rho_m$ of \eqref{pme} and the corresponding pressure variable $p_m = P_m(\rho_m)$ with initial data $\rho_{0,m}$ given by \eqref{matched} that approximate
the initial data \eqref{initial0}.

As we shall see in the lemma below, our choice of initial data $\rho_{0,m}$ will guarantee that
$\rho_m$ is monotonically increasing in
time.  After we obtain convergence result for this particular approximation of $\rho_0$, we can use
$L^1$ contraction for solutions of \eqref{pme} to address the case of general $\rho_{0,m}$.

\begin{lemma}\label{monotone}
$\rho_m$ increases in time for large enough $m$.
\end{lemma}

\begin{proof}
Let us first consider $\tilde\rho_m(x,t) := \rho^E_{0,m}(x) \exp(t G(0)/2)$. Writing $\rho =
\tilde\rho_m$ for the sake of brevity, we can estimate
\begin{align*}
\Delta (\rho^m) + \rho G(p) &= m(m-1) \rho^{m-2} |D\rho|^2 + m \rho^{m-1} \Delta \rho +
\rho G(p)\\
&\geq \rho \left( m \rho^{m-2} \Delta \rho + G(p) \right).
\end{align*}
Due to our assumptions in \eqref{initial}, there exists $m_0$ such that the last term is greater
than $\tilde\rho_m G(0) /2 = \partial_t\tilde\rho_m$ and therefore $\tilde\rho_m$ is a subsolution
of \eqref{pme} for $m \geq m_0$.

Additionally, $\hat \rho_m(x,t) := P_m^{-1}(p_0(x))$ is a stationary subsolution of \eqref{pme}.
We have defined the nondecreasing-in-time functions $\tilde \rho_m$ and $\hat \rho_m$ in such a way that $\max(\tilde \rho_m(\cdot, 0), \hat
\rho_m(\cdot, 0)) = \rho_{0,m}$.
Since a maximum of two subsolutions is also a subsolution, we conclude that
$\rho_m \geq \max(\tilde \rho_m, \hat \rho_m)$, with equality at $t = 0$.
Therefore $\rho_m(\cdot, s) \geq \rho_m(\cdot, 0)$ for any $s \geq 0$.
By the comparison principle we have $\rho_m(\cdot, s) \geq \rho_m(\cdot, t)$ for any $s \geq t$.
\end{proof}

Recall that $p_m = P_m(\rho) := \frac m{m-1} \rho^{m-1}$.
Our goal is to show their convergence to \eqref{FB} as $m\to\infty$. To this end  we first define the \emph{semi-continuous limits} (also referred to as the \emph{half-relaxed limits}) as $m\to\infty$ for a family of functions $f_m$ as
$$
\halfliminf f_m(x,t):= \lim_{r\to 0}  \inf_{\substack{|y|+|s|\leq r\\m\geq r^{-1}}} f_m(x+y, t+s)
$$
and
$$
\halflimsup f_m(x,t):= \lim_{r\to 0}  \sup_{\substack{|y|+|s|\leq r\\m\geq r^{-1}}} f_m(x+y, t+s).
$$

Now let us consider the semi-continuous limits of $\rho_m$ and $p_m$, i.e.,
$$
\rho_1:= \halfliminf \rho_m, \quad p_1 := \halfliminf p_m
$$
and
$$
\rho_2:= \halflimsup \rho_m,\quad \tilde{p}_2:=\halflimsup p_m.
$$
For technical reasons, it is useful to consider a regularization of $\rho_2$ as follows. For a given constant $\sigma>0$ let us define
$$
\rho^\sigma_{m}(x,t):=\sup_{|y-x|\leq \sigma} \rho_m(y,t).
$$
Note that $\rho^{\sigma}$ is a subsolution of \eqref{pme}. Now let us define
$$
\rho^\sigma_2:=  \halflimsup \rho^{\sigma}_m.
$$
Observe that $\rho_1$ is lower semicontinuous and $\rho_2$ and $\rho^{\sigma}_2$ are upper semicontinuous. Let us also define the sets
$$
\Omega_1(t) := \{ p_1(\cdot,t)>0\},\quad \Omega_2(t)= \{\rho_2(\cdot,t)=1\} \quad \hbox{and}\quad \Omega^{\sigma}_2(t) =\{\rho^{\sigma}_2(\cdot,t) =1\},
$$
and define $p^{\sigma}_2(\cdot,t)$ for each $t > 0$ as the smallest supersolution of $-\Delta u = G(u)$ with
Dirichlet boundary data in $\Omega^{\sigma}_2(t)$, that is,
\begin{equation}\label{subsol}
p^{\sigma}_2(x,t) := \inf\{ w(x): w \in C^2(\R^n), -\Delta w > G(w) \hbox{ in a domain containing } \Omega^{\sigma}_2(t), \quad w> 0\},
\end{equation}
and we similarly define $p_2$ corresponding to the set $\Omega_2(t)$.
$p^{\sigma}_2$ is defined in addition to $\tilde{p}_2$ so that we can track the positive set of $p_m$.  $\tilde{p}_2$ is not sufficient for this purpose since we do not know if $p_m$ degenerates to zero
as $m \to \infty$ inside the set $\set{\rho_2 = 1}$. We use the set $\Omega^{\sigma}_2(t)$ instead of $\Omega_2(t)$ to guarantee that the set is regular enough so that the positive set of $p^{\sigma}_2(\cdot,t)$ coincides with the reference set $\Omega^{\sigma}_2(t)$, as we see in the next lemma.  The following lemma shows the relationship
between the various sets, where the last equality is the only nontrivial relation, and explains the
utility of $p^{\sigma}_2$.

\begin{lemma}\label{le:order}
For any $\sigma>0$ we have
$$
\{ p_1 > 0 \} \subset \{ \rho_1 = 1 \}\subset \{\rho^{\sigma}_2=1\} = \overline{\{p^{\sigma}_2>0\}}.
$$
\end{lemma}

\begin{proof}
Suppose that $\rho_1(x_0, t_0) < 1$ for some $(x_0, t_0)$.
Then there exist $m_k, x_k, t_k$, $m_k \to \infty$ and $(x_k, t_k) \to (x_0, t_0)$ as $k\to\infty$
such that $\rho_{m_k}(x_k, t_k) \to \rho_1(x_0, t_0) < 1$.
But then
$p_1(x_0, t_0) \leq \liminf_{k\to\infty} \frac {m_k}{m_k-1} \rho_{m_k}(x_k, t_k)^{m-1} = 0$.
In particular, $\{ p_1 > 0 \} \subset \{ \rho_1 = 1 \}$. The second inclusion in the lemma is due to the fact that $\rho_1\leq \rho_2 \leq \rho^{\sigma}_2$ for any $\sigma>0$. Lastly, note that due to its definition $\Omega^{\sigma}_2(t)$ is closed and has the interior ball property with balls of radius $\sigma$. It now follows from the definition of $p^{\sigma}_2$ that $\overline{\{p^{\sigma}_2(\cdot,t)>0\} }= \Omega^{\sigma}_2(t)$.

\end{proof}

We also point out that due to Lemma~\ref{monotone}, it follows that $\rho_1$ and $\rho_2$ are both
nondecreasing in time.

Let $\omega(\cdot)$ be the mode of continuity for $\rho_0^E$. Below we will show that
\begin{itemize}
\item[(a)] $\tilde{p}_2\leq p^{\sigma}_2$ (Lemma~\ref{le:press-ordered});
\item[(b)] $p_1$ and $p^{\sigma}_2$ are respectively a supersolution of \eqref{FB} with
$\rho^E = \rho_0^E e^{G(0)t}$ and a subsolution of $\rho^{\sigma,E}:= (\rho_0^E + \omega(\sigma))e^{G(0)t}.$ (Theorem~\ref{th:convergence});
\item[(c)] $p_1(\cdot,0) =p_2(\cdot,0)$ is given by  \eqref{1} with $\Omega_0$ (Lemma~\ref{le:initial}).
\end{itemize}
Due to (b) and the stability property of the viscosity solutions of \eqref{FB}, we have $(p_2)_* \leq p_1$. This and (a) yields the convergence results (see Corollary~\ref{cor:main}). We first show that $\Omega^{\sigma}_2(t)$ (and therefore $\Omega_2(t)$, $\Omega_1(t)$) is bounded.

\begin{lem}\label{bounded}
$\Omega^{\sigma}_2(t)$ is bounded for any $t>0$.
\end{lem}

\begin{proof}

By our assumption, $\rho_0$ uniformly converges to zero as $|x|\to\infty$. Therefore  for any $T>0$, there exists $R>0$ such that
$$
\rho^E(x,0) \leq \frac{1}{2}e^{-G(0)T} \hbox{ for } |x|>R.
$$
Let us consider the radial solution $\tilde{\rho}_m$ of \eqref{pme} starting with $\tilde{\rho}_{0,m}$, where $\tilde{\rho}_0$ is given by
$$
\tilde{\rho_0} = \chi_{|x|\leq R} + \frac{1}{2}e^{-G(0)T}\chi_{|x|>R}
$$
and $\tilde{\rho}_{0,m}$ approximates $\tilde{\rho}_0$ as given in Theorem~\ref{thm:perthame:ext}.
Then $\rho_m \leq \tilde{\rho}_m$ by comparison principle for \eqref{pme}. Moreover Theorem~\ref{thm:perthame:ext} yields that
$\tilde\rho_m$ uniformly converges to $\tilde{\rho}$, which solves \eqref{FB} with $\tilde{\rho}^E \leq 1/2$ for $0\leq t\leq T$, and thus has finite propagation property up to $t=T$. Therefore it follows that $\Omega^{\sigma}_2(t)$ is bounded for $0\leq t\leq T$ and we conclude.
\end{proof}






Next we prove the following lemma, to match $\rho_i$'s with $\rho^E$.

\begin{lem}\label{easy}
Let $\rho_1,\rho_2$ be as defined above. Then the following holds:
\begin{itemize}
\item [(a)] $\rho^{\sigma}_2 \leq 1$ for $t \geq 0$ and $\tilde{p}_2\leq M$ for $t > 0$;\\
\item [(b)] $\rho_1\geq \min[1,\rho^E]$ and  $\{\rho^E\geq 1\}\subset\overline{\{p_1>0\}}$;\\
\item [(c)] $\rho^{\sigma}_2 \leq \rho^{\sigma,E}<1$ outside of  $\{\rho^{\sigma}_2= 1\}$.
\end{itemize}
\end{lem}

\medskip

\begin{proof}

1. To show (a), we write $p_m = P_m(\rho_m)$ and $p_{0,m} = P_m(\rho_{0,m})$ and we set $M :=
\sup_{m, x} p_{0,m}(x)$, which is finite by assumption.  By comparison principle, $p_m \leq M$ for
any $m$.  Set
$$
c = - \max_{0 \leq s \leq M} G'(s) > 0.
$$  The function $\phi_m \equiv p_M + (M -
p_M)_+ e^{-c(m-1) t/M}$ is a supersolution of \eqref{pressure} and therefore the comparison
principle yields $p_m \leq \phi_m$ for all $m$.  $\tilde p_2 \leq M$ for $t > 0$ follows.  This
then also implies $\rho_m \leq P_m^{-1}(\phi_m) \leq \max (M, p_M)^{1/{m-1}} \to 1$ for all $t \geq
0$.

\medskip

2.  To show (b), let us choose $(x_0,t_0) \in \R^n\times (0,\infty)$ and $r>0$. To prove the first part,  we will show that
\begin{equation}\label{order}
\rho_1(x,t) \geq \min[1, \rho^E_0(x_0)e^{G(0)t}]-\omega(r)t_0 \hbox{ in } B_r(x_0)\times [0,t_0],
\end{equation}
where $\omega(r)$ is the continuity mode of $\rho_0$ at $x_0$. Since $r>0$ can be arbitrarily chosen and $\omega(r)\to 0$ as $r\to 0+$, we can then conclude. To show \eqref{order} we consider the function
$$
\phi(x,t) = [a(t)\varphi(x) - \omega(r)t]_+ \hbox{ for } 0\leq t\leq t_0,
$$
where $a(t)$ is an increasing $C^1$ function satisfying $a(t)=e^{G(0)t}(\rho_0(x_0)-\omega(r))$ until it hits and $a(t)\equiv 1$ thereafter. $\varphi = \varphi_m$ is a smoothed version of $\chi_{B_r(x_0)}$  that satisfies $-\Delta \varphi^m \leq \e$ for $\e\ll \omega(r)$.  One can for instance use
$$
\varphi(x) := \pth{\frac{\e}{2n} \pth{r^2-\abs{x-x_0}^2}_+}^{1/m}.
$$
We claim that if $m$ is sufficiently large, then  $\phi$ satisfies in its positive set
$$
\begin{aligned}
\phi_t - \Delta \phi^m &\leq \phi_t - \Delta \varphi^m(x)\\
&\leq G(0)\phi-\omega(r) + \e \\
&\leq G(0)\phi - \omega(r)/2 \leq G(p_{\phi})\phi,
\end{aligned}
$$
where $p_{\phi}:=\frac{m}{m-1}\phi^{m-1}$. Note that the first inequality holds since $-\Delta \varphi^m \geq 0$ and $a(t) \leq 1$, and the last inequality holds since $\phi \leq 1-\omega(r)$.

\medskip

Thus $\phi$ is a subsolution of \eqref{pme}, and it follows from the comparison principle of \eqref{pme} that $\phi \leq \rho_m$ and thus $\phi \leq \rho_1$, yielding \eqref{order}.

\medskip

3. Now let us prove the second part of (b)  by modifying the subsolution barrier in the above step. Suppose $\rho^E(\cdot,t_0)\geq 1$ in $B_r(x_0)$ for some $(x_0, t_0)$ and $0 < r < |2G'(0)|^{-1/2}$.
Since $\rho^E$ is non-decreasing in time,
we have $\rho^E \geq 1$ on $\cl B_r(x_0) \times [t_0, \infty)$.
Then from the first part we have $\rho_1 \geq 1$ in $B_r(x_0) \times [t_0, \infty)$,
and thus for any $\delta>0$ and for sufficiently large $m(\delta)$ we have
$$
\rho_m \geq 1-\delta\text{ in }\cl B_r(x_0) \times [t_0, t_1] \hbox{ for } m>m(\delta),
$$
where $t_1 := t_0 + 2 G(0)^{-1} \delta$.

\medskip

Now let us construct the barrier $\phi(x,t) = a(t)\varphi(x)$ to compare with $\rho_m$ in $B_r(x_0)\times [t_0,t_1]$, where $a(t)=e^{(G(0)-3\delta)(t-t_1)}$ and
$$
\varphi(x) = \bra{\frac{\delta}{2n}\pth{r^2 - (x-x_0)^2} +(1-\delta)^m }^{1/m}
$$
so that we have $-\Delta ( \varphi^m ) \leq \delta$ and  $\varphi \geq(1-\delta)$ in $B_r(x_0)$ with equality on $\partial B_r(x_0)$. Also at initial time $t=t_0$, $a(t_0) = e^{(G(0)-3\delta)(t_0-t_1)} < 1-\delta$ since $t_0-t_1 = -2 G(0)^{-1}\delta$. Hence  we have $\varphi \leq 1-\delta \leq \rho_m$ at $t=t_0$ and $\varphi \leq 1-\delta $ on $\partial B_r(x_0)\times [t_0,t_1]$. Also $\phi \geq 1-3\delta \geq \frac 12$ in $B_r(x_0)\times [t_0,t_1]$.

Then we can estimate
$$
\begin{aligned}
\phi_t - \Delta (\phi^m) &\leq \phi_t - \Delta (\varphi^m)\\
&\leq [G(0)-3\delta]\phi + \delta \\
&\leq [G(0)-\delta] \phi \leq G(p_{\phi})\phi,
\end{aligned}
$$
where the first inequality holds due to the fact that $a(t)\leq 1$ and $-\Delta (\varphi^m) \geq 0$, and the last inequality holds for $\delta$ sufficiently small due to the fact that $\phi\geq \frac 12$ and $p_{\phi} G'(0) \geq \delta G'(0) r^2/ n > - \delta / 2$ for large $m$.
Hence we conclude that $\phi\leq \rho_m$ in $B_r(x_0)\times [t_0,t_1]$, which yields
$$
\frac{\delta r^2}{8n}\leq \phi^m \leq p_m \hbox{ in } B_{r/2}(x_0)\times [t_0,t_1]
$$
for $m>m(\delta)$. Thus we conclude that
\begin{equation}\label{conclusion}
p_1(x_0,t_1)=p_1(x_0,t_0+ G(0)^{-1}\delta)>0
\end{equation}
since $r$ is independent of $m$.
As \eqref{conclusion} holds for arbitrarily small $\delta$, it follows that $(x_0,t_0)\in \overline{\{p_1>0\}}$ and we can conclude.

\medskip

4.  Lastly to show (c), we will show that for any given $\delta>0$
\begin{equation}\label{claim}
\rho^{\sigma}_2 \leq \rho^{\sigma,E}\hbox{ on } \{\rho^{\sigma}_2 < 1-2\delta\}.
\end{equation}
We will show this iteratively over time intervals of fixed size $\gamma>0$,
where $\sigma$ satisfies
\begin{equation}\label{delta0}
e^{(G(0)+1)\gamma} (1-\delta) = 1-\delta/2.
\end{equation}

Note that \eqref{claim} holds for $t=0$. Suppose that \eqref{claim} holds up to $t=T$, and let us choose $(x_0,t_0)$ in $\{\rho^{\sigma}_2 < 1-2\delta\} \cap\{T\leq t\leq T+\gamma\}$.
Due to the upper-semicontinuity of $\rho_2$ and its monotonicity in time, there exists $r>0$ such
that $\rho^{\sigma}_2 <1-\delta$ in $\cl B_{2r}(x_0)\times [T,t_0]$. Also note that, due to the first part of
(b)
we have $\min[\rho^E(\cdot, T),1] \leq \rho_1(\cdot, T) \leq \rho^{\sigma}_2(\cdot, T) < 1 - \delta < 1$
on $\cl B_{2r}(x_0)$ and hence
$\rho^{\sigma,E} (\cdot,t_0) =
e^{G(0)(t_0-T)}\rho^E(\cdot, T)\leq e^{G(0) \gamma} (1- \delta) <1 - \delta/2 <1$ in $B_{2r}(x_0)$.

Now based on these facts we will construct a supersolution barrier $\phi$  for \eqref{pme} in $\Sigma:=B_{2r}(x_0)\times [T, t_0]$ such that $\phi \leq \rho^E$ in $ B_r(x_0)\times [T,t_0)$, concluding \eqref{claim}.

Let us choose $\e>0$ and let $\varphi$ be a smooth function in $\cl B_{2r}(x_0)$, $\varphi \leq 1-\delta$ such that $\varphi = 1-\delta$ on $\partial B_{2r}(x_0)$, $\rho^E(\cdot,T)\leq \varphi\leq \rho^E(\cdot,T)+ \e$ in $B_r(x_0)$.  Now consider the barrier
 $$
 \phi(x,t):=e^{(G(0)+\e)t}\varphi(x) \hbox{ in } \Sigma.
 $$
Note that from \eqref{delta0} we have
$\phi \leq  1-\delta/2$ in $\Sigma$, and thus and thus $\phi^m \leq \frac{1}{m^2}$ for large $m$.
Due to this fact and that $\phi$ is smooth, it follows that  $\phi$ is a supersolution of \eqref{pme} in $\Sigma$ for sufficiently large $m$.  Since $\rho^{\sigma}_2 < 1-\delta$ in $B_{2r}(x_0)\times [T,t_0]$, so is $\rho^{\sigma}_m$ for sufficiently large $m$, and thus $\rho^{\sigma}_m \leq \phi$ on the parabolic boundary of $\Sigma$. Thus the comparison principle for \eqref{pme} yields that $\rho^{\sigma}_m\leq \phi$ in $\Sigma$. By sending $\e\to 0$ we conclude that $\rho_2 \leq\rho^{\sigma,E}$ at $(x_0,t_0)$, proving \eqref{claim} for the time interval $[T, T+\gamma]$. Now we conclude \eqref{claim} by iterating our argument over time intervals of length $\gamma$. Lastly we conclude (c) by sending $\delta\to 0$ in \eqref{claim}.

\end{proof}

Next let us prove that $p_2$ is bigger than the limit supremum of $p_m$.

\begin{lem}
\label{le:press-ordered}
$\tilde{p}_2 \leq p_2$.
\end{lem}

\begin{proof}

For any $\e>0$, take a smooth solution $w(x)$ of $-\Delta w \geq G(w)+\e$  with $w\geq \e$ in a domain $U$ containing the closure of $\Omega_2(t_0)$. We will show that $ \tilde p_2(\cdot,t) \leq w$. Then one can conclude by the definition of $p_2$.

\medskip

Due to  Lemma~\ref{easy}(a), $\phi(x,t):= M \frac{t-t_0}{t_1-t_0} +w(x)$ is above $p_m$ on the parabolic boundary of $\Sigma:= U\times [t_1,t_0]$.

Moreover, $\phi$ is a supersolution of \eqref{pressure} for sufficiently large $m$ since
$$
\phi_t +(m-1)\phi(-\Delta \phi - G(\phi)) - |\nabla \phi|^2 \geq -M/(t_1-t_0) +(m-1)\e^2 - |\nabla w|^2 \geq 0 \hbox{ for } m\gg 1.
$$

Thus we conclude that $p_m \leq \phi$ in $\Sigma$, which yields that $\tilde{p}\leq \phi$ in $\Sigma$.

\end{proof}

Now we are ready to show our main claim:

\begin{theorem}\label{th:convergence}
$p_1$ and $p^{\sigma}_2$ are respectively a supersolution of \eqref{FB} with $g=\frac{1}{1-\rho^E}$ and a subsolution of \eqref{FB} with $g^{\sigma}=\frac{1}{1-\rho^{\sigma,E}}$.
\end{theorem}

First note that Lemma~\ref{easy} will allow us to treat the limiting density outside of the maximal density zone essentially as $\rho^E$.

\begin{proof}

1. We will use Definition~\ref{def:visc-barrier}. Let us show the subsolution part first. Suppose $p^{\sigma}_2$ is not a subsolution of \eqref{FB} with $g^{\sigma}$.  This means that there is a {\it superbarrier} $\phi$ of \eqref{FB} with $g^{\sigma}$  in $U:=\{|x-z_0|\leq r\}\times [t_1,t_2]$ which crosses $p^{\sigma}_2$ from above at $t=t_0$: In other words, we have
\begin{itemize}
\item $p^{\sigma}_2 \prec \phi$ on the parabolic boundary of $U$;\\

\item  $p^{\sigma}_2\prec \phi$ in $U\cap\{t_1\leq t< t_0\}$;\\

\item $\sup_{|x-z_0|\leq r}(p^{\sigma}_2-\phi)(\cdot,t) >0$ for $t_0<t<t_2$.\\

\end{itemize}

Since $\phi$ is a superbarrier of \eqref{FB}, there exists $\delta>0$ such that  $\rho^{\sigma,E}
<1-2\delta$ in $B_{\delta}(x_0) \times [t_0 - \delta, t_0 + \delta]$  and
\begin{equation}\label{velocity}
V_{\phi} > \frac{|\nabla \phi|}{1-(\rho^{\sigma,E}+\delta)}  \quad\hbox{ on } \partial\{\phi>0\}\cap \{t\leq t_0\}.
\end{equation}

2. From its definition, $p^{\sigma}_2$ cannot cross $\phi$ before its support crosses that of $\phi$.
It follows that $\chi_{\{p^{\sigma}_2>0\}}(\cdot,t_0)$ crosses $\chi_{\{\phi>0\}}$ at $t=t_0$, and thus
along a subsequence $\rho^{\sigma}_m \geq\chi_{\{\phi>0\}} + \chi_{\{\phi=0\}}(\rho^{\sigma,E} + \delta)$ for the
first time  at $(x_m,t_m)$ with $t_m\to t_1 \leq t_0$ as $m\to\infty$. Note that the crossing point
exist since $\rho_m$ is continuous in time.

Let $x_0$ be a limit point of $\{x_m\}$. If $\phi(x_0,t_1)>0$ then we have a contradiction since in
that case it can be easily checked that $\phi$ is a supersolution of \eqref{pressure} in a
neighborhood of $(x_0,t_1)$ for sufficiently large $m$. Also due to Lemma 4.4 (c) and the fact
that, from Lemma~\ref{le:order},
$$
\{\rho^{\sigma}_2=1\}= \overline{\{p^{\sigma}_2(\cdot,t)>0\}}\subset\{\phi(\cdot,t)>0\} \hbox{ for } t<t_0,
$$

the limit point $(x_0,t_1)$
cannot be outside of $\overline{\{\phi>0\}}$.  Hence $(x_0,t_1)$ lies on $\partial\{\phi>0\}$, and
$t_1=t_0$.

\medskip

Relying on the continuity of $\rho^E$, let us choose $0<r<\delta$ such that
\begin{equation}\label{density_0}
\rho^{\sigma,E} \leq \rho^{\sigma,E}(x_0,t_0)+\frac{\delta}{2} \leq \rho^{\sigma,E} + \delta\hbox{ in }D:=B_r(x_0)\times [t_0-r,t_0+r].
\end{equation}
  We now localize $\phi$ in $D$ to a radial profile. Since $|D\phi|\neq 0$ on $\partial\{\phi>0\}$, it follows from the regularity of $\phi$ that $\partial\{\phi(\cdot,t_0)>0\}$ is a $C^2$ surface. Therefore we can choose $r$ in above definition of $D$ small enough such that there is a exterior ball $B_{r/2}(y_0)$ in $\{\phi(\cdot,t_0)=0\}$ touching $x_0$ on its boundary. We can then use the Taylor expansion of $\phi$ up to the second order in space and first order in time to construct
  a new radial superbarrier $\varphi(x,t) = \varphi(|x-y_0|,t)$ of \eqref{FB} in $D$ satisfying \eqref{velocity} such that $\{\varphi(\cdot,t_0)=0\}=B_{r/2}(y_0)$ and $\varphi >\phi$ on the parabolic boundary of $D$.
 Then, replacing $\varphi$ with $\varphi(x,t)$ for sufficiently small $\e>0$ if necessary,  $p^\sigma_m:= P_m(\rho_m^{\sigma})$ crosses $\varphi(x,t)$ in $D$ for large $m$.
Due to Lemma~\ref{easy}(c) and \eqref{density_0},
\begin{equation}\label{density}
\rho_{\varphi}:= \chi_{\{\varphi>0\}} + (\bar{\rho}^{\sigma,E}-1)> \rho^{\sigma}_m \hbox{ on the parabolic boundary of } D
\end{equation}
for large $m$, where $\bar{\rho}^{\sigma,E}:= \rho^{\sigma,E}(x_0,t_0)+\frac{\delta}{2}$.

\medskip

Now let $\tilde{\rho}_m:=\rho_{\varphi,m}$ be the corresponding solution of \eqref{pme} in $D$, with fixed data $1$ on $\partial B_r(x_0)$ with approximating initial data given as in \eqref{rhom0} in section \ref{se:convergence radial}.  Note that, due to the comparison principle of \eqref{pme}, $\rho^{\sigma}_m \leq \tilde{\rho}_m$ in $D$.   On the other hand, the solution $(p,\bar{\rho}^{\sigma,E})$ of \eqref{FB} in $D$ satisfies $p \prec \varphi$ in $D$ due to \eqref{velocity}. Due to Theorem~\ref{thm:perthame:ext} $\limsup_{m\to\infty} \tilde{\rho}_m =\bar{\rho}^{\sigma,E} <\rho^{\sigma,E}+\delta$ outside of the support of $p$ in $D$, in particular in the zero set of $\varphi$ in $D$.  This contradicts the fact that $\rho^{\sigma}_m$ crosses $\chi_{\{\varphi>0\}} + \chi_{\{\varphi=0\}}(\rho^{\sigma,E} +\delta)$ in $D$. We can now conclude.

\medskip

3. For the supersolution part, first note that  the requirement $\{\rho^E\geq 1\}\subset \overline{\{p_1>0\}}$ is satisfied by Lemma~\ref{easy}(b). Next suppose a {\it subbarrier}  $\phi$ of \eqref{FB} crosses $p_1$ from below in $\{|x-z_0| \geq r\}\times [t_1,t_2]$ at $t=t_0$. Parallel arguments as above using Lemma~\ref{easy}(b) would yield the conclusion.

\end{proof}

Lastly, to apply  comparison principle for $p_1$ and $p_2$, we show that the initial data for $\rho_i$'s and $p_i$'s respectively coincide.

\begin{lem}\label{le:initial}
At $t=0$ we have

\begin{itemize}
\item[(a)]$\lim_{t\to 0^+}\rho_i(\cdot,t)= \rho_0 := \rho_0^E \chi_{\Omega_0^c} + \chi_{\Omega_0}$ locally uniformly away from $\partial\Omega_0$;\\
\item[(b)] $\lim_{t\to 0^+} p_i(\cdot,t) = p_0$ uniformly;
\end{itemize}
where $p_0$ is the unique solution of $-\Delta p=G(p)$ in $\Omega_0$ with zero boundary data on $\partial\Omega_0$.
\end{lem}

\medskip

\begin{proof}
1. Let us first show $(a)$. First of all note that $\rho^E$ converges uniformly to $\rho_0$ away from $\Omega_0=\{\rho_0=1\}$. Also note that, from their definition, $\Omega_2^{\sigma}(t)$ converges to $\Omega_2(t)$ in Hausdorff distance as $\sigma\to 0$.

  Hence
by Lemma~\ref{easy} we have
\begin{equation}\label{easy2}
\rho^E = \rho_1=\rho_2 \hbox{ outside of } \{\rho_2=1\}.
\end{equation}

Moreover, by Lemma~\ref{easy} we have $\rho_1\geq 1$ on $\Omega_0$. Hence it is enough to show that
\begin{equation}\label{claim00}
\{\rho_2=1\}\cap \{t=0\} = \overline{\Omega_0} \times \{ t = 0 \}.
\end{equation}
To this end we consider the domain
$$
\Omega_\e:=\{x: d(x,\Omega_0)\leq 3\e\}
$$
for a given $\e>0$, and choose a point $x_0\in\partial\Omega_\e$.  By our assumption there exists $\delta>0$ depending on $\e$ such that  $\rho_0\leq 1-2\delta$ in $B_{2\e}(x_0)$, and thus
\begin{equation}\label{surrounding_density}
\rho^E \leq 1-\delta\hbox{ in } B_{2\e}(x_0)\times [0,t_1]\hbox{ for some }t_1>0.
\end{equation}
 Let us now consider the radial function $\phi(x,t)$ in $B_{2\e}(x_0)-B_{\e(t)}(x_0)$  such that $\phi=0$ on $\partial B_{\e(t)}(x_0)$, $\phi=1$ on $B_{2\e}(x_0)$ and
$$
-\Delta \phi(x) = G(0) \hbox{ in } B_{2\e}(x_0)-B_{\e(t)}(x_0).
$$

Note that we have $|D\phi| \leq M/\e$ on $\partial B_{\e(t)}(x_0)$ where $M$ is independent of $\e$ as long as $\e(t) \geq \e/2$. Combining this fact and \eqref{surrounding_density}, it follows that
if we choose $\e(t)=(\e-\frac{M}{\e\delta} t)$ and $\rho^E_{\phi}(0)=1-2\delta$, then  $(\phi, \rho^E_{\phi})$ is a supersolution of \eqref{FB} in $B_{2\e}(x_0)\times [0, t_\e]$, where $t_\e = \min[\frac{\e^2\delta}{M},t_1]$. This and Theorem~\ref{thm:perthame:ext} yields that
$$
\rho_2 \leq \rho_{\phi}^E <1\hbox{ in } B_{\e/2}(x_0)\times [0,t_\e].
$$
This concludes \eqref{claim00} and therefore (a).

\medskip

2. Next we prove (b). To this end we need to ensure that $p_m$ does not vanish inside of $\Omega_0$. Again this follows from Theorem~\ref{thm:perthame:ext}, since at each interior point $x_0\in \Omega_0$ with $B_r(x_0)\in\Omega_0$ for some $r>0$ we can consider radial solution of \eqref{FB} with $\rho^E_{\phi}=0$ and apply Theorem~\ref{thm:perthame:ext} to show that the corresponding solutions $\tilde{p}_m$ of \eqref{pressure} uniformly converges to $\phi$. Now we can conclude since $p_m \geq \tilde{p}_m$ by the comparison principle of \eqref{pressure}.

\medskip

3. Now we are ready to prove (b). Fix $\e>0$ and define
$$
\Omega_f := \{x: \dist(x, \R^n \setminus \Omega_0) > \e \}\hbox{ and }\Omega_g:=\Omega_\e=\{x:\dist(x,\Omega_0) \leq \e\}.
$$
In view of  (a) and step 2.,  there exist $\delta = \delta(\e)>0,  t_0=t_0(\e)>0$ and $M$ such that
for $m>M$ and $0\leq t\leq t_0$ the following holds:  $p_m \leq \delta$ on $\partial\Omega_g$ , $p_m \geq \delta$  in
$\Omega_f$. Let us consider $f$ and $g$ defined by

\begin{equation*}
- \Delta f = G(f)-\e \hbox{ in } \Omega_f  \ \text{ and} \  f=\delta\hbox{ on }\partial\Omega_f,
\end{equation*}
and
\begin{equation*}
 -\Delta g = G(g)+\e\hbox{ in }\Omega_g  \ \text{ and} \  g=\delta\hbox{ on }\partial\Omega_g.
\end{equation*}

%
Let
$$
\phi(x,t):= a(t)f(x) \hbox{ and } \psi(x,t) := b(t) g(x),
$$
where
$$
  a(t):= \min [\delta e^{\frac{m}{2}\e t} ,1] \hbox{ and } b(t) := \max[\delta^{-1}e^{-\frac{m}{2}\e\delta t},1].
$$
 Note that the gradient of $f$ is bounded from above in $\Omega_f$. Using this fact, direct calculations then yield that for sufficiently large $m$ and $\phi$ and $\psi$  are respectively subsolution and supersoluton to \eqref{pressure} in $\Omega_f\times (0,t_0]$ and $\Omega_g \times (0,t_0]$.   Thus the comparison principle for \eqref{pressure} and the choice of $\delta$ and $t_0$ yield
$$
\psi\leq p_m \hbox{ in } \Omega_g \times [0,t_0] \hbox{ and } p_m \leq \phi\hbox{ in }\Omega_f \times [0,t_0].
$$
Letting $m\to\infty$ and using arbitrarily small $\e>0$,  we conclude that the  $p_m$'s converge  uniformly  to the solution of the elliptic equation $\Omega_0$ with zero boundary data.

\smallskip

\end{proof}

\bigskip

Theorem~\ref{th:convergence} and Lemma~\ref{le:initial} together yield our main result:

\begin{cor}\label{cor:main}
Let $p$ be the unique lower-semicontinuous viscosity solution of \eqref{FB} given by Theorem~\ref{th:hs-well-posedness}. Then the following holds as $m\to\infty$:

\begin{itemize}
\item[(a)] $\halflimsup p_m = p^*$ and $\halfliminf p_m = p_*$.
\item[(b)] $\rho_m$ uniformly converges to $\rho:= \chi_{\{p>0\}} + \rho^E \chi_{\set{p = 0}}$ away from $\partial\{p>0\}$.\\
\end{itemize}
\end{cor}

\begin{proof}
From Theorem ~\ref{th:convergence} and the stability property of viscosity solutions of \eqref{FB}, it follows that $\bar{p}:=(\liminf_{\sigma\to 0} p_2^{\sigma})_*$ is a supersolution of \eqref{FB} with $g = \frac{1}{1-\rho^E}$. Due to Lemma~\ref{le:initial} (a) and the convergence of $\Omega^{\sigma}_2(0)$ to $\Omega_0$ in Hausdorff distance, we conclude that $\bar{p}(\cdot,t)$ uniformly converges to $p_0(\cdot,0)$ as $t\to 0$.

From the comparison principle Theorem~\ref{th:comparisonhsg} it follows that $\bar{p} \leq p_1$. Since $p_2 \leq p_2^{\sigma}$ for any $\sigma>0$, it follows that  $(p_2)_* \leq \bar{p}\leq  p_1$. Since $p_1\leq p_2$ by definition, this means $(p_1)^* = (p_2)^*$ and $(p_1)_*=(p_2)_*$. This yields that $p=p_1=(p_2)_*$ is a viscosity solution of \eqref{FB} with surrounding density $\rho^E$, and this yields (b). The convergence of $\rho_m$ in the interior of $\{p>0\}$ then follows from (b).

\medskip

It remains to show that $\rho_m$ converges to $\rho^E$ away from $\{p>0\}$. Note that due to
Lemma~\ref{le:order}
$$
\overline{\{p>0\}}=\overline{\{p_2>0\}} = \{\rho_2=1\}.
$$
This and Lemma~\ref{easy} (c) yields that $\halflimsup_{m\to\infty}\rho_m =\rho_2\leq \rho^E$ away
from $\overline{\{p>0\}}$. Now we conclude by Lemma~\ref{easy}(b), which says $\halfliminf_{m\to\infty} \rho_m =\rho_1\geq \min[1,\rho^E]$.

\end{proof}

Recall that an ``almost'' contraction property is available for any two solutions $\rho_m$, $\hat
\rho_m$ of \eqref{pme} from \cite[(2.12)]{PQV} in the form
\begin{equation}\label{contraction}
\norm{\rho_m(t) - \hat \rho_m(t)}_1 \leq e^{G(0) t} \norm{\rho_m(0) - \hat\rho_m(0)}_1 \quad \hbox{ for any } t>0.
\end{equation}

Using the above formula as well as the uniform convergence result obtained in Corollary~\ref{cor:main} and Corollary~\ref{cor:00}, we have the following convergence result for general approximating initial data $\rho_{0,m}$:

\begin{cor}\label{cor:general}
Let $\rho_0 := \chi_{\Omega_0} + \rho^E_0 \chi_{\Omega_0^c}$ with $\Omega_0$, $\rho^E_0$ as given in
\eqref{initial0}, with Lipschitz continuous $\rho_0^E$. Suppose that $\rho_{0,m}$ converge to $\rho_0$ in $L^1(\R^n)$. Then the corresponding solution $\rho_m$ of \eqref{pme} with the initial data $\rho_{0,m}$ converges to $\rho$ as given in Corollary~\ref{cor:main} in the following sense:
$$
\norm{\rho_m(t) - \rho(t)}_1 \to 0 \hbox{ as } m\to \infty \hbox{ for a.e. } t>0.
$$
\end{cor}

\section{A BV estimate on the positivity set of the pressure}

Here we show that $\partial\{p(\cdot,t)>0\}$ has finite perimeter as long as $\rho^E$ stays strictly less than $1$ near $\partial\{p(\cdot,t)>0\}$. The result already follows from the BV estimates in \cite{PQV}, however our proof is based on geometric arguments and thus is of independent interest.
\begin{lem}\label{lem:0}
Let $\Omega_t(p):= \{p(\cdot,t)>0\}$, where $p$ is as given in Corollary~\ref{cor:main},  and assume that $\rho^E <1$ on $\partial\Omega(t)$.  Then for given $r>0$, there exists sets $\Omega_{r,t}$  such that
$$
\Omega_{r,t} \subset \Omega_t(p) \hbox{ for each } t>0
$$
such that
\begin{itemize}
\item[(a)] $\Omega_{r,t}$ increases with respect to $r$;\\
\item[(b)] $\Omega_{r,t}$ has interior ball properties with radius $r$;\\
\item[(c)]$|\Omega_{r,t} - \Omega_t(p)| \leq Ce^{t}r$.
\end{itemize}
\end{lem}

\begin{proof}
To prove this, take the initial positive set
$$
\Omega^r_0:= \{x: d(x, \Omega_0^c) \geq 2r\}
$$
and consider the corresponding approximating solution $\rho_{m,r}$ of \eqref{pme} with its limiting
initial density $$\rho_{0,r}:= \chi_{\Omega_{r,0}} + \rho^E\chi_{\Omega_{r,0}^C}.$$ Let us now take
$\Omega_2^{\sigma}(t)$ and $p_2^{\sigma}$ as defined in \eqref{subsol} with $\rho_{m,r}$ instead of
$\rho_m$. Let us choose now $\sigma=r$. Then due to \eqref{th:convergence}, $p_2^r$ is a subsolution of \eqref{FB} with $g$ and $\Omega_2^r(0) = \{ x: d(x,\Omega^r_0) \leq r\}\subset \Omega_0$. Hence by comparison principle of \eqref{FB} we have $p_2^r\leq p$, and thus
 $$
 \Omega_2^r(t) = \{p_2^r(\cdot,t)>0\} \subset \Omega_t(p)
 $$ for all $t>0$. (c) follows from the contraction inequality ~\ref{contraction} applied to $\rho_m$ and $\rho_{m,r}$ given in Corollary~\ref{cor:main} in the limit $m\to\infty$.  \end{proof}

\begin{proposition}\label{BV}
 Under the same assumptions as in Lemma~\ref{lem:0}, For any $r>0$, $\Omega_{r,t}$ has uniformly bounded perimeter. As a consequence $\{p(\cdot,t)>0\}$ is a set of finite perimeter.
\end{proposition}

\begin{proof}
We consider $\Omega^n_t:= \Omega_{r_n,t}$ with $r_n = 2^{-n}$. We claim that for $r\leq r_n$ there is at most $C_dr^{1-d}$ balls of radius $r$ covering the boundary of $\Omega_{r_n,t}$.

We will only show the claim for $r=r_n$, For smaller radius $r<r_n$, the claim holds due to Lemma 2.5 of \cite{ACM}.
We know that $\Omega^n_t$ increases with respect to $n$ with
\begin{equation}\label{first}
|\Omega^n_t - \Omega^{n+1}_t | \leq Cr_n,
\end{equation}
where $C$ is independent of $n$. Moreover, from the construction above, in fact we have the following relation between $\Omega^n_t$ and $\Omega^{n+1}_t$:
\begin{equation}\label{second}
\{x: d(x, \Omega^n_t) \leq cr_{n+1}\} \subset \Omega^{n+1}_t.
\end{equation}
where $c$ is independent of the choice of $n$.

\medskip

Now let us take an open covering $\mathcal{O}$ of the boundary of $\Omega^{n+1}_t$ consisting of balls of radius $r_{n+1}$ with its center on a boundary point. Let's take out a family of disjoint balls in $\mathcal{O}$ obtained by  Vitali's covering Lemma. In each of this disjoint balls, at least one third of the ball is taken by the interior of $\Omega^{n+1}_t$ by the interior ball property satisfied at the center of each ball. Also due to \eqref{second} at least a fixed portion of this interior is away from $\Omega^n_t$. Now we conclude that if the number of the disjoint balls are $N$, then \eqref{first} yields that $N(r_{n+1})^d \leq Cr_{n+1}$, or
$$
N\leq  C(r_{n+1})^{1-d}.
$$

Hence we conclude.
\end{proof}

 \end{document}